\RequirePackage{fix-cm}
\documentclass[smallextended]{svjour3}       
\smartqed  
\usepackage{graphicx}
\usepackage{comment}
\usepackage[utf8]{inputenc} 
\usepackage[T1]{fontenc}    
\usepackage{hyperref}       
\usepackage{url}            
\usepackage{booktabs}       
\usepackage{amsmath}
\usepackage{lipsum}
\allowdisplaybreaks

\usepackage{amsthm}
\usepackage{amsfonts} 
\usepackage{multirow}
\usepackage{nicefrac} 
\usepackage{mathtools}
\DeclarePairedDelimiter\abs{\lvert}{\rvert}
\DeclarePairedDelimiter\norm{\lVert}{\rVert}
\usepackage{microtype}      
\usepackage{xcolor}
\usepackage{graphicx}
\usepackage{multicol}
\graphicspath{ {./figs/} }
\usepackage{stmaryrd}
\usepackage{bm}
\usepackage[normalem]{ulem}
\usepackage[ruled,vlined]{algorithm2e}
\usepackage{caption}
\usepackage{subcaption}
\usepackage[toc]{appendix}
\usepackage{enumerate}
\usepackage{lineno}
\usepackage{pst-pdf}

\usepackage{multirow}
\usepackage{booktabs}
\usepackage{caption}

\usepackage{diagbox}
\usepackage{array}

\usepackage{orcidlink}

\numberwithin{equation}{section}
\newtheoremstyle{iremark}
  {\topsep}   
  {\topsep}   
  {\upshape}  
  {0pt}       
  {\itshape}  
  {.}         
  {5pt plus 1pt minus 1pt} 
  {\thmname{#1}\thmnumber{ \itshape#2}\thmnote{ (#3)}} 

\theoremstyle{iremark}
\newenvironment{nalign}{
    \begin{equation}
    \begin{aligned}
}{
    \end{aligned}
    \end{equation}
    \ignorespacesafterend
}

\hypersetup{
    colorlinks=true, 
    linkcolor=blue, 
    urlcolor=red, 
    citecolor=magenta,
    linktoc=all 
}

\newcommand{\spc}{{\hspace{0.5cm}}}

\newcommand{\jmh}{{j-\frac{1}{2}}}
\newcommand{\jph}{{j+\frac{1}{2}}}
\newcommand{\half}{\frac{1}{2}}

\usepackage{caption}
\usepackage{subcaption}

\newcommand*\dif{\mathop{}\!\mathrm{d}}

\newcommand{\bs}{\boldsymbol}
\newcommand{\cblue}[1]{\textcolor{black}{#1}}

\newcommand{\iph}{i+\frac{1}{2}}
\newcommand{\imh}{i-\frac{1}{2}}

\newenvironment{frcseries}{\fontfamily{frc}\selectfont}{}

\newcommand{\no}{\nonumber}

\begin{document}

\title{A positivity-preserving  second-order scheme for multi-dimensional systems of non-local conservation laws
\thanks{This work was done while one of the authors, G D Veerappa
Gowda, was a Raja Ramanna Fellow at TIFR-Centre for Applicable Mathematics, Bangalore. Nikhil Manoj gratefully acknowledges the financial support from the Council of Scientific and Industrial Research (CSIR), Government of India, in the form of a doctoral fellowship.
}
}
\subtitle{}

\titlerunning{Second-order scheme for systems of non-local conservation laws}        

\author{Nikhil Manoj\,\orcidlink{0009-0003-6812-8633}\and G. D. Veerappa Gowda \orcidlink{0009-0009-8525-4790}        \and Sudarshan Kumar K. \orcidlink{0009-0005-1186-4167}
}

\institute{
Nikhil Manoj \at
              School of Mathematics\\
Indian Institute of Science Education and Research\\
Thiruvananthapuram, India-695551\\
\email{nikhilmanoj2020@iisertvm.ac.in} \and
G. D. Veerappa Gowda \at
              Centre for Applicable Mathematics\\
Tata Institute of Fundamental Research\\
Bangalore, India - 560065\\
\email{gowda@tifrbng.res.in}           
           \and
           Sudarshan Kumar K. \at
              School of Mathematics\\
Indian Institute of Science Education and Research\\
Thiruvananthapuram, India-695551\\
\email{sudarshan@iisertvm.ac.in} 
}

\date{Received: date / Accepted: date}

\maketitle

\begin{abstract}
   In this article, we present and analyze a fully discrete second-order scheme for a general class of non-local system of conservation laws  in multiple spatial dimensions. The method employs a MUSCL-type spatial reconstruction coupled with Runge-Kutta time integration. We analytically prove that the proposed scheme preserves the positivity in all the {unknowns}, a critical property for ensuring the physical validity of quantities like density, which must remain non-negative. Additionally, the scheme is proven to exhibit $\mathrm{L}^\infty$-stability. Numerical experiments conducted on both the non-local scalar and system cases illustrate the importance of the second-order scheme when compared to its first-order counterpart and verify the theoretical results.

\end{abstract}

\section{Introduction}
Non-local conservation laws have emerged as a vital tool in modeling various physical phenomena, including traffic flow \cite{bayen2021,blandin2016,chiarello-globalentropy,chiarellotosin2023,ciotirfayad2021,SopasakisKatsoulakis2006}, crowd dynamics \cite{aggarwal2016b,burgergoatin2020,colomboGaravelloMercier2012,ColomboHertyMercier2011,ColomboMercier2012,colomborossi2018,colomborossi_2019},  structured population dynamics in biology \cite{Perthame_book2007}, sedimentation \cite{betancourt2011}, supply chains \cite{armbruster2006continuum} and conveyor belts \cite{Gottlich2014,Rossi2019WellposednessOA}. \cblue{ In many of these applications, the inclusion of non-local terms in the flux functions offers a more precise framework for capturing interactions between densities, such as those in crowd dynamics or traffic flow models. In this study, our focus is on a class of system of  non-local conservation laws in several space dimensions. For the one-dimensional case, non-local conservation laws have been well-studied in the literature from both theoretical and numerical points of view, for example, see \cite{aggarwal_holden_vaidya2024,aggarwalima2024,aggarwalvaidya2025,goatin2016,kuang2024,keimer2017}.}  However, their extension to multiple space dimensions is comparatively less explored, with only a limited number of results addressing its well-posedness.  \cblue{For instance, the authors in \cite{aggarwal2015}  proved the existence of a weak solution for a general system in two dimensions by establishing the convergence of a dimensionally split scheme with Lax-Friedrichs numerical flux.}  \cblue{Additionally}, the existence and uniqueness of measure-valued solutions to a class of multi-dimensional {\cblue {problems}} were analyzed in \cite{crippa2013}. Local-in-time existence and uniqueness results for certain multi-dimensional non-local equations under weak differentiability assumptions on the convolution kernel were recently studied in \cite{colombo2024multidimensional}. Furthermore, the error analysis of  first-order finite volume schemes for a one-dimensional problem was presented in \cite{aggarwal_holden_vaidya2024}, and its extension to the  multi-dimensional case was also discussed.  In this work, we are interested in the general system of multi-dimensional non-local conservation laws \cblue{treated} in \cite{aggarwal2015}. 
\par
Although  first-order \cblue{numerical} methods are reliable and aid in  ensuring well-posedness of the underlying problems, second- and higher-order methods offer substantially improved accuracy, particularly for two and three-dimensional problems. This has led to an increasing emphasis on research aimed at developing  high-order methods.  In the context of \cblue{non-local conservation laws, for} one-dimensional problems, convergence results are available for second-order schemes.  For example, the convergence of a second-order scheme to the entropy solution for a class of one-dimensional problems was analyzed in \cite{gowda2023}. Also, see \cite{burgerconterasvillada2023,chiarello2019multiclass,GoatinScialanga} for more numerical results in this direction. Furthermore, high-order DG and CWENO schemes were discussed for the one-dimensional case in \cite{chalons2018} and \cite{friedrich2019CWENO}, respectively. \cblue{ To the best of our knowledge, so far, no results are  available on second- or high-order schemes for the multi-dimensional case.}
\par
It is the purpose of this work to propose a fully-discrete  second-order scheme  for systems of multi-dimensional non-local conservation laws {\cblue{ and present numerical simulations together with desirable theoretical results.}}
To derive a second-order scheme, we combine a MUSCL-type spatial reconstruction \cite{van1979towards} with a second-order strong stability preserving Runge-Kutta (RK) time-stepping method \cite{gottlieb1998otalVD,gottliebsigalshu2001} . As a key contribution of this work, we show that the resulting scheme satisfies the  positivity-preserving property.   \cblue{ This property is particularly important in models such as those of crowd dynamics, where the unknowns represent densities of different species and must remain non-negative.}  Additionally, we establish that the numerical solutions obtained from the proposed second-order scheme are $\mathrm{L}^{\infty}$- stable. These analytical results are validated through numerical examples. We also examine the numerical convergence of the second-order scheme using numerical experiments and highlight its significance. \cblue{Furthermore, the asymptotic compatibility of the proposed scheme is numerically investigated in the context of the singular limit problem (see \cite{colombocrippa2023,colombocrippa2019,kuang2024}), as the non-local horizon parameter tends to zero.} We are also interested in the theoretical convergence of the second-order scheme, for which the main ingredient is the bounded variation (BV) estimates. We note that, for the case of local conservation laws as well, no BV estimates are available; instead, convergence is established in \cite{Coquel1991} through weak BV estimates for measure-valued solutions. Given the difficulties associated with obtaining BV estimates, we aim to further investigate in this direction in a forthcoming paper. 

{\cblue {The rest of this paper is organized as follows. In Section \ref{sec:equations} we outline the non-local system of conservation laws under consideration.  Section \ref{section:foschem} describes a first-order finite volume scheme using the Lax-Friedrichs type numerical fluxes and details the discretization of the convolution terms. In Section \ref{section:soscheme}, we present the second-order numerical scheme. The positivity-preserving property of the proposed scheme is established in Section \ref{section:positivity}. In Section \ref{section:Linfinty}, the $\mathrm{L}^\infty$-  stability of  the second-order scheme is proven.  Numerical examples are given in Section \ref{section:numericalexp}}, to illustrate the performance of the proposed second-order scheme. The conclusions are summarized in Section \ref{section:conclusion}.}

\section{System of non-local conservation laws}\label{sec:equations}
We consider the system of non-local conservation laws in $n$ \cblue{space} dimensions \cblue{studied} in \cite{aggarwal2015}: 
\begin{equation}\label{eq:system}
    \partial_t \boldsymbol{\rho} + \nabla_{\boldsymbol{x}}\cdot \boldsymbol{F}(t,\boldsymbol{x},\boldsymbol{\rho},\boldsymbol{\eta}_{1} \ast \boldsymbol{\rho}, \cdots, \boldsymbol{\eta}_{n} \ast \boldsymbol{\rho}) = 0,
\end{equation}
where $ \boldsymbol{x} := \left(x_{1}, x_{2}, \cdots, x_{n}\right)$ and the unknown is 
\begin{align*}    \boldsymbol{\rho} &:= \left(\rho^{1}, \rho^{2}, \cdots, \rho^{N}\right),
\end{align*}

and for each fixed $r \in \{1,2, \dots, n\},$ the convolution kernel corresponding to the $r$-th dimension is given by the $m \times N$ matrix
\begin{align*}
\boldsymbol{\eta}_{r} &:= \begin{pmatrix}
\eta_{r}^{1,1} & \cdots & \eta^{1,N}_{r} \\
\vdots & \ddots & \vdots \\
\eta^{m,1}_{r} & \cdots & \eta^{m,N}_{r}
\end{pmatrix},
\end{align*}
where   $\eta^{l,k}_{r}:\mathbb{R}^{n} \to \mathbb{R}.$ 
Also,  \eqref{eq:system} is posed along  with the initial condition 
\begin{align}\label{eq:ic_system}
    \bs \rho(\bs x,0) = \bs \rho_0(\bs x).
\end{align}
For the sake of simplicity of the exposition, we restrict our attention to the case of systems of non-local conservation laws in two dimensions, i.e, $n=2$ and  $\boldsymbol{x} = (x,y).$ Note that all the results in this case can be easily extended to the case of general $n$-dimensional systems.
The convolution kernel functions corresponding to the $x$-and $y$-direction are then given  by \cblue{the matrices}
\begin{align*}\boldsymbol{\eta}&:=\boldsymbol{\eta}_{1} =  \begin{pmatrix}
\eta^{1,1} & \cdots & \eta^{1,N} \\
\vdots & \ddots & \vdots \\
\eta^{m,1} & \cdots & \eta^{m,N}
\end{pmatrix} \quad \mbox{and} \quad   \boldsymbol{\nu}:=\boldsymbol{\eta}_{2} = \begin{pmatrix}
\nu^{1,1} & \cdots & \nu^{1,N} \\
\vdots & \ddots & \vdots \\
\nu^{m,1} & \cdots & \nu^{m,N}
\end{pmatrix},
\end{align*}
respectively, and the flux function takes the form
\begin{align*}
    \boldsymbol{F}(t,\boldsymbol{x},\boldsymbol{\rho},\boldsymbol{\eta} \ast \boldsymbol{\rho}, \boldsymbol{\nu} \ast \boldsymbol{\rho})&:=   \begin{pmatrix}
f^{1}(t,x,y,\rho^{1},\boldsymbol{\eta} \ast \boldsymbol{\rho})& g^{1}(t,x,y,\rho^{1}, \boldsymbol{\nu} \ast \boldsymbol{\rho}) \\
\vdots & \vdots \\
f^{N}(t,x,y,\rho^{N},\boldsymbol{\eta} \ast \boldsymbol{\rho}) & g^{N}(t,x,y,\rho^{N}, \boldsymbol{\nu} \ast \boldsymbol{\rho})
\end{pmatrix}^{T}.
\end{align*}
For $k \in \{1, \dots, N \},$ we  now focus on  the problem  associated with the $k$-th unknown $\rho^{k}$ of \eqref{eq:system}, given by 
\begin{nalign}\label{2d-nonlocalgeneraleqn}
    \partial_{t}\rho^{k} + \partial_{x}f^{k}+ \partial_{y}g^{k} &= 0, \,\, t>0,\,\, (x, y) \in \mathbb{R}^{2},\\
    \boldsymbol{\rho}(0, x, y) &= \left(\rho_{0}^{k}(x, y)\right)_{k=1}^{N}, \, (x, y) \in \mathbb{R}^{2},
\end{nalign}
where $f^{k}=f^{k}(t, x, y,\rho^{k}, \boldsymbol{\eta}*\boldsymbol{\rho}),\, g^{k}=g^{k}(t, x, y, \rho^{k}, \boldsymbol{\nu}*\boldsymbol{\rho})$ 
and the convolution terms  are defined as
\begin{align*}
\left(\boldsymbol{\eta}*\boldsymbol{\rho}\right)_{q}(t,x,y) &\coloneqq \sum_{k=1}^{N}\int\int_{\mathbb{R}^{2}} \eta^{q,k}(x- x^{\prime}, y- y^{\prime}) \rho^{k}(t, x^{\prime}, y^{\prime}) \dif x^{\prime} \dif y^{\prime}, \\
\left(\boldsymbol{\nu}*\boldsymbol{\rho}\right)_{q}(t,x,y) &\coloneqq \sum_{k=1}^{N}\int\int_{\mathbb{R}^{2}} \nu^{q,k}(x- x^{\prime}, y- y^{\prime}) \rho^{k}(t, x^{\prime}, y^{\prime}) \dif x^{\prime} \dif y^{\prime},
\end{align*}
for $q \in \{1,2, \dots, m\}.$
\subsection{Notations}
\cblue{In what follows, we denote $\mathbb{R}_+:=[0,\infty), $  $\mathbb{R}_+^N:=[0,\infty)^N$ and $\norm{\cdot} :=\norm{\cdot}_{\mathrm{L}^{\infty} }.$ For a vector-valued quantity $\boldsymbol{\rho}: \mathbb{R}^2 \to \mathbb{R}^N$, we define $\displaystyle \norm{\bs\rho} := \max_{k \in \{1,2, \dots,N\}}\norm{\rho^k}$ and $\displaystyle \norm{\boldsymbol{\rho}}_{\mathrm{L}^1}:= \sum_{k=1}^N \norm{\rho^k}_{\mathrm{L}^1}.$ Also, for a matrix-valued quantity $\boldsymbol{\eta}: \mathbb{R}^2 \to \mathbb{R}^{m \times N},$  we define  $\displaystyle\norm{\partial_x \boldsymbol{\eta}}:= \max_{q,k}\norm{\partial_x \eta^{q,k}}$ and $\displaystyle \norm{\partial_y \boldsymbol{\eta}}:= \max_{q,k}\norm{\partial_y \eta^{q,k}}.$  Further, for any $c,d \in \mathbb{R}$, we define $\mathcal{I}(c,d):= \bigl(\min\{c,d\}, \max\{c,d\}\bigr)$ and for vectors $\boldsymbol{A_1}, \boldsymbol{A_2} \in \mathbb{R}^m,$ $\mathcal{I}(\boldsymbol{A_1},\boldsymbol{A_2}):= \{\kappa\boldsymbol{A_1}+(1-\kappa)\boldsymbol{A_2} \,| \,       \kappa \in (0,1)\}.$} 
\subsection{Hypotheses}
In this work, the non-local problem \eqref{eq:system}, \eqref{eq:ic_system} is studied under the following hypotheses:\\
\textbf{(\hypertarget{H0}{H0})} For all $t \in \mathbb{R}_{+}, (x, y) \in \mathbb{R}^{2}$ and $\boldsymbol{A}, \boldsymbol{B} \in \mathbb{R}^{m},$
\begin{enumerate}
    \item $f^{k}, g^{k} \in \textrm{C}^{2}(\mathbb{R}_{+} \times \mathbb{R}^{2} \times \mathbb{R} \times \mathbb{R}^{m}; \mathbb{R}),\quad$
    $\partial_{\rho}f^{k}, \partial_{\rho}g^{k} \in \mathrm{L}^{\infty}(\mathbb{R}_{+} \times \mathbb{R}^{2} \times \mathbb{R} \times \mathbb{R}^{m}; \mathbb{R})$ \,\, and
    \item $f^{k}(t, x, y, 0, \boldsymbol{A}) = g^{k}(t, x, y, 0, \boldsymbol{B})= 0,$ 
\end{enumerate} for $k \in \{1,2, \dots, N\}.$\\
\textbf{(\hypertarget{H1}{H1})}  There exists an $M>0$ such that for all $t, x, y, \rho, \boldsymbol{A} \, \mbox{and}\, \boldsymbol{B}$ in the respective domains
\begin{align*}
    &\lvert \partial_{x}f^{k}  \rvert,  \norm{\nabla_{A}f^{k}} \leq M\lvert \rho \rvert \quad \mbox{and} \quad \lvert \partial_{y}g^{k}  \rvert,  \norm{\nabla_{B}g^{k}}   \leq M\lvert \rho \rvert \quad \, \mbox{for} \, k \in \{1,2, \dots, N\}.
\end{align*}
\textbf{(\hypertarget{H2}{H2})} $ \boldsymbol{\eta}, \boldsymbol{\nu} \in (\mathrm{C}^{2}\cap \cblue{\mathrm{W}^{1, \infty})}(\mathbb{R}^{2}; \mathbb{R}^{m\times N}).$

We note that, under these assumptions, along with some additional hypotheses and suitable CFL conditions, the existence of a weak solution to problem \eqref{eq:system}, \eqref{eq:ic_system} was proved in \cite{aggarwal2015} through the convergence of a first-order numerical scheme employing a Lax-Friedrichs type numerical flux (see Theorem 2.3, \cite{aggarwal2015}).

\section{First-order scheme}\label{section:foschem}
In this section, we describe the construction of a first-order finite volume scheme with Lax-Friedrichs type numerical flux to approximate \eqref{2d-nonlocalgeneraleqn}. We discretize the spatial domain into Cartesian grids with mesh sizes $\Delta x$ and $\Delta y$ in the $x$ and $y$ directions, respectively, as follows
\begin{align*}
    x_{i}&:= i\Delta x, \, y_{j}:= j\Delta y, \,
    x_{\iph}:= (\iph)\Delta x \,\, \mbox{and} \,\, y_{\jph}:= (\jph)\Delta y, \,\, \forall \, i,j \in \mathbb{Z}. 
\end{align*}
Now, the discretization of the spatial domain is given by $\mathbb{R}^2 = \bigcup_{i,j \in \mathbb{Z}} [x_{i-\frac{1}{2}}, x_{i+\frac{1}{2}})\\ \times [y_{j-\frac{1}{2}}, y_{j+\frac{1}{2}}),$  where $x_{i+\frac{1}{2}}- x_{i-\frac{1}{2}}= \Delta x$ and $y_{j+\frac{1}{2}}- y_{j-\frac{1}{2}}= \Delta y.$ The time domain is discretized using a time-step $\Delta t$ and we denote $t^{n}= n\Delta t \ \mbox{for} \ n \in \mathbb{N}.$ We also denote $\displaystyle \lambda_{x}:= \frac{\Delta t}{\Delta x}\, \, \mbox{and} \,\, \lambda_{y}:= \frac{\Delta t}{\Delta y}.$ The initial datum $\boldsymbol{\rho}_{0}$ is discretized as
\begin{align*}
    \rho_{ij}^{k,0}= \frac{1}{\Delta x \Delta y} \int_{x_{i-\frac{1}{2}}}^{x_{i+\frac{1}{2}}}\int_{y_{j-\frac{1}{2}}}^{y_{j+\frac{1}{2}}}\rho_{0}^{k}(x, y) \dif x \dif y \quad \mbox{for} \ i,j \in \mathbb{Z}.
\end{align*}
A first-order finite volume approximation for \eqref{2d-nonlocalgeneraleqn} can be written as
\begin{nalign}\label{eq:foscheme}
    \rho^{k,n+1}_{ij} &= \rho^{k,n}_{ij} - \lambda_{x}\left[F^{k,n}_{i+\frac{1}{2}, j}(\rho_{i,j}^{k,n}, \rho_{i+1,j}^{k,n}) - F^{k,n}_{i-\frac{1}{2}, j}(\rho_{i-1,j}^{k,n}, \rho_{i,j}^{k,n})\right]\\
    &\spc - \lambda_{y} \left[G^{k,n}_{i, j+\frac{1}{2}}(\rho_{i,j}^{k,n}, \rho_{i, j+1}^{k,n}) - G^{k,n}_{i, j-\frac{1}{2}}(\rho_{i,j-1}^{k,n}, \rho_{i, j}^{k,n})\right],
\end{nalign} {\cblue{where the cell-interface numerical fluxes are defined using the Lax-Friedrichs flux, as discussed in }} \cite{aggarwal2015}:
\begin{nalign}\label{LxF-2Dflux}
F^{k,n}_{i+\frac{1}{2}, j}(u, v) := \frac{f^{k,n}_{i+\frac{1}{2}, j}(u) + f^{k,n}_{i+\frac{1}{2}, j}(v)}{2} - \frac{\alpha(v-u)}{2 \lambda_{x}},\\
G^{k,n}_{i, j+\frac{1}{2}}(u, v) := \frac{g^{k,n}_{i, j+\frac{1}{2}}(u) + g^{k,n}_{i, j+\frac{1}{2}}(v)}{2} - \frac{\beta(v-u)}{2 \lambda_{y}},
\end{nalign}
for fixed $\alpha, \beta$ \cblue{which will be specified later (see Remark \ref{rem:alphabeta})}, where
\begin{align*}
    f^{k,n}_{i+\frac{1}{2}, j}(\rho) \coloneqq f^{k}(t^{n}, x_{i+\frac{1}{2}},y_{j},\rho, \boldsymbol{A}^{n}_{i+\frac{1}{2}, j}), \,\,
    g^{k,n}_{i, j+\frac{1}{2}}(\rho) \coloneqq g^{k}(t^{n}, x_{i},y_{j+\frac{1}{2}},\rho, \boldsymbol{B}^{n}_{i, j+\frac{1}{2}}).
\end{align*}
\par
Here, the terms $\boldsymbol{A}^{n}_{i+\frac{1}{2}, j}:= \left(A^{q,n}_{i+\frac{1}{2}, j}\right)_{q=1}^{m}$ and $\boldsymbol{B}^{n}_{i, j+\frac{1}{2}} := \left(B^{q,n}_{i, j+\frac{1}{2}}\right)_{q=1}^{m}$ are approximations of the convolution terms in the sense that for $q=1,2, \dots, m,$ $ A^{q,n}_{i+\frac{1}{2}, j} \approx \left(\boldsymbol{\rho}*\boldsymbol{\eta}\right)_{q}(t^{n},x_{i+\frac{1}{2}}, y_{j})$ and $B^{q,n}_{i, j+\frac{1}{2}} \approx (\boldsymbol{\rho}*\boldsymbol{\nu})_{q}(t^{n},x_{i}, y_{j+\frac{1}{2}}).$ These approximations are derived using the midpoint quadrature rule as described below:
\begin{nalign}\label{eq:A}
&\left(\boldsymbol{\rho}*\boldsymbol{\eta}\right)_{q}(t^{n},x_{i+\frac{1}{2}}, y_{j}) \\ & \spc = \sum_{k=1}^{N} \int\int_{\mathbb{R}^{2}} \eta^{q,k}(x_{i+\frac{1}{2}}- x^{\prime}, y_{j}- y^{\prime}) \rho^{k}(t, x^{\prime}, y^{\prime}) \dif x^{\prime} \dif y^{\prime} \\
& \spc = \sum_{k=1}^{N}\sum_{l,p \in \mathbb{Z}}\int_{x_{l-\frac{1}{2}}}^{x_{l+\frac{1}{2}}}\int_{y_{p-\frac{1}{2}}}^{y_{p+\frac{1}{2}}} \eta^{q,k}(x_{\iph}- x^{\prime}, y_{j}- y^{\prime}) \rho^{k}(t, x^{\prime}, y^{\prime}) \dif x^{\prime} \dif y^{\prime} \\
& \spc \approx
\Delta x \Delta y  \sum_{k=1}^{N}\left[\sum_{p,l \in \mathbb{Z}}\eta^{q,k}(x_{i+\frac{1}{2}}-x_{l}, y_{j}-y_{p})        \ \rho^{k,n}_{l, p}\right]\\
&\spc = 
\Delta x \Delta y  \sum_{k=1}^{N}\left[\sum_{p,l \in \mathbb{Z}}\eta^{q,k}_{i+\frac{1}{2}-l, j-p}        \ \rho^{k,n}_{l, p}\right] =: A^{q,n}_{i+\frac{1}{2}, j} 
\end{nalign}
and 
\begin{align*}
&(\boldsymbol{\rho}*\boldsymbol{\nu})_{q}(t^{n},x_{i}, y_{j+\frac{1}{2}}) \\& \spc = \sum_{k=1}^{N} \int\int_{\mathbb{R}^{2}} \nu^{q,k}(x_{i}- x^{\prime}, y_{j+\frac{1}{2}}- y^{\prime}) \rho^{k}(t^{n}, x^{\prime}, y^{\prime}) \dif x^{\prime} \dif y^{\prime} \\
& \spc =  \sum_{k=1}^{N} \sum_{l,p \in \mathbb{Z}}\int_{x_{l-\frac{1}{2}}}^{x_{l+\frac{1}{2}}}\int_{y_{p-\frac{1}{2}}}^{y_{p+\frac{1}{2}}} \nu^{q,k}(x_{i}- x^{\prime}, y_{j+\frac{1}{2}}- y^{\prime}) \rho^{k}(t^{n}, x^{\prime}, y^{\prime}) \dif x^{\prime} \dif y^{\prime} \\
& \spc \approx
 \Delta x \Delta y  \sum_{k=1}^{N}\left[\sum_{l,p \in \mathbb{Z}}\nu^{q,k}_{i-l, j+\frac{1}{2}-p} \ \rho^{k,n}_{l, p}\right] =: B^{q,n}_{i, j+\frac{1}{2}},
\end{align*}
with the notation
$\eta^{q,k}_{i+\frac{1}{2}, j} := \eta^{q,k}(x_{i+\frac{1}{2}}, y_{j}) \ \mbox{and} \ \nu^{q,k}_{i, j+\frac{1}{2}} := \nu^{q,k}(x_{i}, y_{j+\frac{1}{2}}).$
Finally, the approximate solution is given by the  piecewise constant function $\boldsymbol{\rho}_{\Delta} := \left(\rho^{1}_\Delta, \rho^{2}_\Delta, \cdots, \rho^{N}_\Delta\right),$ where $\rho_{\Delta}^{k}(t, x, y)$ is defined by 
\begin{align*}
    \rho_{\Delta}^{k}(t, x, y) = \rho_{ij}^{k,n} \,\, \mbox{for} \,\, (t,x,y) \in [t^{n}, t^{n+1}) \times [x_{i-\frac{1}{2}}, x_{i+\frac{1}{2}})\times [y_{j-\frac{1}{2}}, y_{\jph}),
\end{align*} $\mbox{for} \ n \in \mathbb{N}, \,\, \ i,j \in \mathbb{Z}$ and $k \in \{1,2, \dots, N\}.$

\cblue{\begin{remark}
The convergence of the first-order scheme \eqref{eq:foscheme} can be established using arguments similar to those for dimensionally split first-order schemes in \cite{aggarwal2015}.
\end{remark}}
\section{Second-order scheme}\label{section:soscheme}
To develop a second-order scheme, we adhere to the fundamental principle of utilizing a spatial linear reconstruction and a Runge-Kutta time-stepping method. Specifically, we employ a two-stage Runge-Kutta method, where in each step,  a piecewise linear polynomial is reconstructed  within each cell using slopes obtained from the minmod limiter in each direction. Additionally, the reconstructed piecewise polynomial is formulated to preserve the cell average in each cell. 
To begin with,  we describe the reconstruction procedure at the time level $t^n,$ where we write the piecewise linear polynomial in each cell as $$\tilde{\rho}^{k,n}_{\Delta}(x,y):= ax+by+c, \quad \mbox{for} \quad (x,y) \in [x_{\imh}, x_{\iph})\times [y_{\jmh}, y_{\jph}),$$  where $a,b$ and $c$ are constants.
Given that $\tilde{\rho}^{k,n}_\Delta$ preserves the cell averages,  we obtain
\begin{nalign}\label{eq:recon_stage1}
    \tilde{\rho}^{k,n}_{\Delta}(x,y)
    = \rho_{ij}^{k,n} + a(x-x_{i})+ b(y-y_{j}),
\end{nalign} where $a=\partial_x\tilde{\rho}^{k,n}_{\Delta}(x_i,y_j)$ and $b= \partial_{y}\tilde{\rho}^{k,n}_{\Delta}(x_i,y_j).$ The slopes  are determined using the minmod slope-limiter as $\Delta x\partial_x\tilde{\rho}^{k,n}_{\Delta}(x_i,y_j) =  \sigma^{x,k,n}_{ij},$ $\Delta y\partial_y\tilde{\rho}^{k,n}_{\Delta}(x_i,y_j) = \sigma^{y,k,n}_{ij},$ where  
\begin{nalign}\label{eq:slopes}
     \sigma^{x,k,n}_{ij} &:= 2 \theta \textrm{minmod}\left((\rho^{k,n}_{i,j}-\rho^{k,n}_{i-1,j}), \  \frac{1}{2}(\rho^{k,n}_{i+1, j}-\rho^{k,n}_{i-1,j}), \  (\rho^{k,n}_{i+1,j}-\rho^{k,n}_{i,j})\right), \\
    \sigma^{y,k,n}_{ij} &:= 2 \theta \textrm{minmod}\left((\rho^{k,n}_{i,j}-\rho^{k,n}_{i, j-1}), \  \frac{1}{2}(\rho^{k,n}_{i,j+1}-\rho^{k,n}_{i,j-1}), \  (\rho^{k,n}_{i,j+1}-\rho^{k,n}_{i,j})\right),
\end{nalign} 
$ \mbox{for} \ \theta \in [0,1],$ where the minmod function is  defined by
\begin{align*}
\textrm{minmod}(a_{1}, \cdots , a_{m}) \coloneqq \begin{cases}\mathrm{sgn}(a_{1}) \min\limits_{1 \leq k \leq  m}\{\lvert a_{k}\rvert\} \hspace{0.55cm} \mbox{if} \ \mathrm{sgn}(a_{1}) = \cdots = \mathrm{sgn}(a_{m}),\\
0 \hspace{3.45cm}\mbox{otherwise.}
\end{cases}
\end{align*}
The face values of the reconstructed polynomial in the $x$-direction are given by
\begin{nalign}\label{eq:facevalues_x}
\rho_{i+\frac{1}{2},j}^{k,n,-}= \rho^{k,n}_{i,j}+\frac{\sigma_{i,j}^{x,k,n}}{2}, \ \ \ \rho_{i-\frac{1}{2}, j}^{k,n,+}= \rho^{k,n}_{i,j}-\frac{\sigma_{i,j}^{x,k,n}}{2}.
\end{nalign}
Similarly, the face values  in the $y$-direction are given by
\begin{nalign}\label{eq:facevalues_y}
 \rho_{i,j+\frac{1}{2}}^{k,n,-}= \rho^{k,n}_{i,j}+\frac{\sigma_{i,j}^{y,k,n}}{2}, \ \ \ \rho_{i, j-\frac{1}{2}}^{k,n,+}= \rho^{k,n}_{i,j}-\frac{\sigma_{i,j}^{y,k,n}}{2},
\end{nalign}
where, within each cell, the superscripts $+$ and $-$  indicate the left (bottom) and right (top) interfaces, respectively.
\par
Given the cell-averaged solutions $\rho_{\Delta}^{k,n}, \, k \in \{1,2,\dots,N\}$ at the time stage $t^n$, the fully discrete scheme involves two stages of the Runge-Kutta method \cite{gottlieb1998otalVD,shu1988} to compute the solution at the time level $t^{n+1}.$ This is  described as follows.\\
{\bf Step 1:} Define
\begin{nalign}\label{rk-2d-1stEulerstep}
    \rho^{k,(1)}_{ij} &= \rho^{k,n}_{ij} - \lambda_{x}\left[F^{k,n}_{i+\frac{1}{2}, j}(\rho_{i+\frac{1}{2},j}^{k,n,-}, \rho_{i+\frac{1}{2},j}^{k,n, +}) - F^{k,n}_{i-\frac{1}{2}, j}(\rho_{i-\frac{1}{2},j}^{k,n,-}, \rho_{i-\frac{1}{2},j}^{k,n,+})\right] \\
    &\spc -\lambda_{y}\left[G^{k,n}_{i, j+\frac{1}{2}}(\rho_{i,j+\frac{1}{2}}^{k,n, -}\rho_{i,j+\frac{1}{2}}^{n, +}) - G^{k,n}_{i, j-\frac{1}{2}}(\rho_{i,j-\frac{1}{2}}^{k,n,-}, \rho_{i,j-\frac{1}{2}}^{k,n,+})\right],
\end{nalign} $\mbox{for each }\;i,\; j\in \mathbb{Z},$ \cblue{where the numerical fluxes $F$ and $G$ are as defined in \eqref{LxF-2Dflux}, with $ \alpha, \beta \in (0,\frac{1}{3(1+\theta)})$.}
Next, reconstruct the piecewise linear polynomial from the values $\rho^{k,(1)}_{ij}$ as in \eqref{eq:recon_stage1} and compute the face values $\rho^{k,(1),\pm}_{\iph,j}$ and $\rho^{k,(1),\pm}_{i,\jph}$ following \eqref{eq:facevalues_x} and \eqref{eq:facevalues_y}.  \\
{\bf Step 2:} Define 
\begin{nalign}\label{rk-2d-2stEulerstep}
    \rho^{k,(2)}_{ij} &= \rho^{k,(1)}_{ij} - \lambda_{x}\left[F^{k,(1)}_{i+\frac{1}{2}, j}(\rho_{i+\frac{1}{2},j}^{k,(1),-}, \rho_{i+\frac{1}{2},j}^{k,(1), +}) - F^{k,(1)}_{i-\frac{1}{2}, j}(\rho_{i-\frac{1}{2},j}^{k,(1),-}, \rho_{i-\frac{1}{2},j}^{k,(1),+})\right]\\
    &\spc-\lambda_{y}\left[G^{k,(1)}_{i, j+\frac{1}{2}}(\rho_{i,j+\frac{1}{2}}^{k,(1), -}\rho_{i,j+\frac{1}{2}}^{k,(1),+}) - G^{k,(1)}_{i, j-\frac{1}{2}}(\rho_{i,j-\frac{1}{2}}^{k,(1),-}, \rho_{i,j-\frac{1}{2}}^{k,(1),+})\right], 
\end{nalign} 
$\mbox{for each }\;i,\; j\in \mathbb{Z}.$ Finally, the solution at the $(n+1)$-th time-level is now computed as
\begin{align}\label{RK-2d}
\rho_{ij}^{k,n+1}=\frac{\rho^{k,n}_{ij}+\rho^{k,(2)}_{ij}}{2}
\end{align} and  for $k \in \{1,2, \dots, N\},$ we write the approximate solution corresponding to the second-order scheme \eqref{RK-2d} as \begin{align*}
    \rho_{\Delta}^{k}(t, x, y) = \rho_{ij}^{k,n} \quad \mbox{for} \quad (t,x,y) \in [t^{n}, t^{n+1}) \times [x_{i-\frac{1}{2}}, x_{i+\frac{1}{2}})\times [y_{j-\frac{1}{2}}, y_{\jph})
\end{align*} 
$\mbox{for} \ n \in \mathbb{N} \ \mbox{and} \ i,j \in \mathbb{Z}.$

\begin{remark}
     In the slope limiter \eqref{eq:slopes}, $\theta = 0$ corresponds to first-order spatial accuracy, while $\theta = 0.5$ recovers the standard minmod limiter, achieving second-order spatial accuracy.
\end{remark}

\section{Positivity-preserving property}\label{section:positivity}
We now show that the second-order scheme given by \eqref{RK-2d} admits a positivity-preserving property, i.e., for $n\in \mathbb{N}\cup \{0\},$ $\rho_{ij}^{k,n+1} \ge 0$ whenever $\rho_{ij}^{k,n} \ge 0.$ To begin with, we write the Euler forward step \eqref{rk-2d-1stEulerstep}
as the average
\begin{align}\label{eq:euler_as_v+w}
\rho^{k,(1)}_{ij} =\frac{V_{ij}^{k,(1)}+W^{k,(1)}_{ij}}{2},
\end{align} 
where
\begin{align}\label{dimsplitx-1}
V_{ij}^{k,(1)}& := \rho^{k,n}_{ij} - 2\lambda_{x}\left[F^{k,n}_{i+\frac{1}{2}, j}(\rho_{i+\frac{1}{2},j}^{k,n,-}, \rho_{i+\frac{1}{2},j}^{k, n, +}) - F^{k,n}_{i-\frac{1}{2}, j}(\rho_{i-\frac{1}{2},j}^{k,n,-}, \rho_{i-\frac{1}{2},j}^{k,n,+})\right], 
\end{align}
and
\begin{align*}
W_{ij}^{k,(1)}&:= \rho^{k,n}_{ij}-2\lambda_{y}\left[G^{k,n}_{i, j+\frac{1}{2}}(\rho_{i,j+\frac{1}{2}}^{k,n, -}\rho_{i,j+\frac{1}{2}}^{k,n, +}) - G^{k,n}_{i, j-\frac{1}{2}}(\rho_{i,j-\frac{1}{2}}^{k,n,-}, \rho_{i,j-\frac{1}{2}}^{k,n,+})\right]. 
\end{align*}
\cblue{Also, we note a useful property of the minmod reconstruction in the following remark.}

    \begin{remark}\label{remark:2}
    For given $k$ and $n,$ if $\rho^{k,n}_{i,j}\ge 0$ $\forall\,\, i,j \in \mathbb{Z}$ then it follows that $|\rho_{i+\frac{1}{2},j}^{k,n,-}-\rho_{i-\frac{1}{2},j}^{k,n,+}|\le 2\theta \rho_{ij}^{k,n}.$ This can be verified in the following lines.
     From the definition of slopes in \eqref{eq:slopes}, we obtain
\begin{align*}
    0 \le \frac{(\rho_{i+\frac{1}{2},j}^{k,n,-}-\rho_{i-\frac{1}{2},j}^{k,n,+})}{\rho_{i,j}^{k,n}-\rho_{i-1,j}^{k,n}},\;\;    \frac{(\rho_{i+\frac{1}{2},j}^{k,n,-}-\rho_{i-\frac{1}{2},j}^{k,n,+})}{\rho_{i+1,j}^{k,n}-\rho_{i,j}^{k,n}}  \le 2 \theta.  
\end{align*}
\par
Additionally, we observe that either $\rho_{i-1,j}^{k,n} <\rho_{i,j}^{k,n}$ or $\rho_{i+1j}^{k,n} <\rho_{i,j}^{k,n}$ provided $\abs{\rho_{i+\frac{1}{2},j}^{k,n,-}-\rho_{i-\frac{1}{2},j}^{k,n,+}} \neq 0.$
Splitting this in to two cases  and using the assumption $\rho_{i,j}^{k,n} \ge 0,$ we obtain\\
Case 1: If $\rho_{i,j}^{k,n} > \rho_{i+1,j}^{k,n}$ then 
$$\abs{\rho_{i+\frac{1}{2},j}^{k,n,-}-\rho_{i-\frac{1}{2},j}^{k,n,+}}=\frac{(\rho_{i+\frac{1}{2},j}^{k,n,-}-\rho_{i-\frac{1}{2},j}^{k,n,+})}{(\rho_{i+1,j}^{k,n}-\rho_{i,j}^{k,n})}|\rho_{i+1,j}^{k,n}-\rho_{i,j}^{k,n}| \le 2 \theta |\rho_{i+1,j}^{k,n}-\rho_{i,j}^{k,n}| \le 2\theta \rho_{ij}^{k,n}.$$
Case 2: If $\rho_{i,j}^{k,n} > \rho_{i-1,j}^{k,n}$  then 
$$\abs{\rho_{i+\frac{1}{2},j}^{k,n,-}-\rho_{i-\frac{1}{2},j}^{k,n,+}}=\frac{(\rho_{i+\frac{1}{2},j}^{k,n,-}-\rho_{i-\frac{1}{2},j}^{k,n,+})}{(\rho_{i,j}^{k,n}-\rho_{i-1,j}^{k,n})}|\rho_{i,j}^{k,n}-\rho_{i-1,j}^{k,n}| \le 2 \theta|\rho_{i,j}^{k,n}-\rho_{i-1,j}^{k,n}| \le 2\theta \rho_{ij}^{k,n}. $$ 
This shows that 
$|\rho_{i+\frac{1}{2},j}^{k,n,-}-\rho_{i-\frac{1}{2},j}^{k,n,+}|\le 2\theta \rho_{ij}^{k,n}.$
\end{remark}

\begin{theorem}\label{thm:positivity}
Assume that the hypotheses \hyperlink{H0}{\textbf{(H0)}}, \hyperlink{H1}{\textbf{(H1)}} and \hyperlink{H2}{\textbf{(H2)}} hold and for all $k \in \{1,2, \dots, N\}$ the time-step $\Delta t$ satisfies the following CFL conditions
\begin{nalign}\label{cfl-2D}
   \bar{\lambda}_{x} &\leq \frac{\min\{1, 4- 6\bar{\alpha} (1+\theta), 6\bar{\alpha}\}}{\bigl(6(1+\theta)\lVert\partial_{\rho}f^{k} \rVert  +1\bigr)} , \quad  \bar{\lambda}_{y} \le \frac{\min\{1, 4- 6\bar{\beta} (1+\theta), 6\bar{\beta}\}}{\bigl(6(1+\theta)\lVert\partial_{\rho}g^{k} \rVert  +1\bigr)},
 \end{nalign}
where $\bar{\alpha}:=2\alpha,\;\,\bar{\beta}:=2 \beta,\,\,  \bar{\lambda}_x:=2 \lambda_{x}, \,\,\bar{\lambda}_y:=2 \lambda_{y}$ and the parameter $\theta \in [0,1]$ is as defined in the minmod \cblue{slope-limiter \eqref{eq:slopes}.}
Additionally,  assume that the mesh sizes are sufficiently small so that $\displaystyle \Delta x, \Delta y \leq \frac{1}{3M}$ where $M$ is as in \hyperlink{H1}{\textbf{(H1)}}. If the initial datum $\boldsymbol{\rho}_{0}$ is such that $\boldsymbol{\rho}_{0} \in \mathrm{L}^{1} \cap \mathrm{L}^{\infty}(\mathbb{R}^{2}; \mathbb{R}_{+}^{N})$, then the approximate solution $\boldsymbol{\rho}_{\Delta}$ given by the second-order scheme (\ref{RK-2d}) satisfies $\rho_{\Delta}^{k}(t, x, y) \geq 0$ for all $k \in \{1,2, \dots, N\},$  $t \in \mathbb{R}_{+}$ and $(x, y) \in \mathbb{R}^{2}$. 
\end{theorem}
\begin{proof}
To prove the positivity of the second-order scheme, we employ induction on the time index $n$. The base case for $n=0$ holds trivially as initial data is non-negative, i.e., $\rho_{ij}^{k,0} \geq 0$ for all $i, j \in \mathbb{Z}$ and for all $k \in \{1,2, \dots, N\}.$
For $n\ge 0,$ it is required to show that
$\rho_{ij}^{k,n+1} \ge 0$ whenever $\rho_{ij}^{k,n} \ge 0.$ To do this, it suffices to prove that the forward Euler step \eqref{rk-2d-1stEulerstep} satisfies $\rho_{ij}^{k,(1)}\ge 0$   whenever $\rho_{ij}^{k,n} \ge 0$. This reduces to verifying that  $V_{ij}^{k,(1)} \ge 0,$ as  the  same argument applies to  $W_{ij}^{k,(1)}.$
\par
By adding and subtracting the term $\bar{\lambda}_{x}\left(F^{k,n}_{i+\frac{1}{2}, j}(\rho_{i+\frac{1}{2},j}^{k,n,-}, \rho_{i-\frac{1}{2},j}^{k,n,+})- \right. \\\left.F^{k,n}_{i-\frac{1}{2}, j}(\rho_{i+\frac{1}{2},j}^{k,n,-}, \rho_{i-\frac{1}{2},j}^{k,n, +})\right)$ in \eqref{dimsplitx-1},  $V_{ij}^{k,(1)}$ reads as 
\begin{nalign}\label{eq:RKxsweep}
        V_{ij}^{k,(1)} 
   &   =    \rho_{i, j}^{k,n} - a^{k,n}_{i-\frac{1}{2},j} (\rho_{i,j}^{k,n}-\rho_{i-1,j}^{k,n}) + b^{k,n}_{i+\frac{1}{2}, j}(\rho^{k,n}_{i+1,j} - \rho^{k,n}_{i,j}) \\
   & \hspace{0.5cm}- \bar{\lambda}_{x}\Big(F^{k,n}_{i+\frac{1}{2}, j}(\rho_{i+\frac{1}{2},j}^{k,n,-}, \rho_{i-\frac{1}{2},j}^{k,n,+}) - F^{k,n}_{i-\frac{1}{2}, j}(\rho_{i+\frac{1}{2},j}^{k,n,-}, \rho_{i-\frac{1}{2},j}^{k,n,+})\Big)\\
   & = \Bigl(1-a^{k,n}_{i-\frac{1}{2},j}- b^{k,n}_{i+\frac{1}{2},j} \Big)\rho_{i, j}^{k,n}+a^{k,n}_{i-\frac{1}{2},j}\rho_{i-1, j}^{k,n} +b^{k,n}_{i+\frac{1}{2},j} \rho_{i+1, j}^{k,n} \\
    &\hspace{0.5cm}- \bar{\lambda}_{x}\Big(F^{k,n}_{i+\frac{1}{2}, j}(\rho_{i+\frac{1}{2},j}^{k,n,-}, \rho_{i-\frac{1}{2},j}^{k,n,+}) - F^{k,n}_{i-\frac{1}{2}, j}(\rho_{i+\frac{1}{2},j}^{k,n,-}, \rho_{i-\frac{1}{2},j}^{k,n,+})\Big),
\end{nalign}
where 
\begin{align*}
    a_{i-\frac{1}{2},j}^{k,n} &\coloneqq  \bar{\lambda}_{x}\tilde{a}_{i-\frac{1}{2},j}^{k,n} \left(\frac{\rho_{i+\frac{1}{2},j}^{k,n, -}-\rho_{i-\frac{1}{2},j}^{k,n,-}}{\rho_{i,j}^{k,n}-\rho_{i-1,j}^{k,n}} \right),\\
    b_{i+\frac{1}{2},j}^{k,n} &\coloneqq - \bar{\lambda}_{x}\tilde{b}_{i+\frac{1}{2},j}^{k,n} \left(\frac{\rho_{i+\frac{1}{2},j}^{k,n, +}-\rho_{i-\frac{1}{2},j}^{k,n,+}}{\rho_{i+1,j}^{k,n}-\rho_{i,j}^{k,n}} \right),
\end{align*}
with \begin{align*}
    \tilde{a}_{i-\frac{1}{2},j}^{k,n} &:= \frac{\left[F^{k,n}_{i-\frac{1}{2}, j}(\rho_{i+\frac{1}{2},j}^{k,n,-}, \rho_{i-\frac{1}{2},j}^{k,n, +}) - F^{k,n}_{i-\frac{1}{2}, j}(\rho_{i-\frac{1}{2},j}^{k,n,-}, \rho_{i-\frac{1}{2},j}^{k,n,+})\right]}{(\rho_{i+\frac{1}{2},j}^{k,n, -}-\rho_{i-\frac{1}{2},j}^{k,n,-})} \quad \mbox{and}\\
    \tilde{b}_{i+\frac{1}{2},j}^{k,n} &:= \frac{\left[F^{k,n}_{i+\frac{1}{2}, j}(\rho_{i+\frac{1}{2},j}^{k,n,-}, \rho_{i+\frac{1}{2},j}^{k,n, +}) - F^{k,n}_{i+\frac{1}{2}, j}(\rho_{i+\frac{1}{2},j}^{k,n,-}, \rho_{i-\frac{1}{2},j}^{k,n,+})\right]}{(\rho_{i+\frac{1}{2},j}^{k,n, +}-\rho_{i-\frac{1}{2},j}^{k,n,+})}.
\end{align*}
\par
We will now show that $\displaystyle 0 \le  a_{i-\frac{1}{2},j}^{k,n}, b^{k,n}_{i+\frac{1}{2},j} \leq \frac{1}{3}.$ Observe that
$$0 \leq \left(\frac{\rho_{i+\frac{1}{2},j}^{k,n, -}-\rho_{i-\frac{1}{2},j}^{k,n,-}}{\rho_{i,j}^{k,n}-\rho_{i-1,j}^{k,n}  } \right), \left(\frac{\rho_{i+\frac{1}{2},j}^{k,n, +}-\rho_{i-\frac{1}{2},j}^{k,n,+}}{\rho_{i,j}^{k,n}-\rho_{i-1,j}^{k,n}  } \right) \leq (1+\theta),$$
where $\theta$ is as defined in the minmod limiter in \eqref{eq:slopes}.
From the definition of $F_{i+\frac{1}{2}, j}^{k,n}$ in \eqref{LxF-2Dflux}  and applying the mean value theorem, it follows that 
\begin{nalign}\label{eq:a_upperbd}
     a_{i-\frac{1}{2},j}^{k,n}&= \frac{\bar{\lambda}_{x}\hat{a}_{i-\frac{1}{2},j}^{k,n}}{2(\rho^{k,n,-}_{i +\frac{1}{2},j}- \rho^{k,n,-}_{i -\frac{1}{2}, j})} \left(\frac{\rho_{i+\frac{1}{2},j}^{k,n, -}-\rho_{i-\frac{1}{2},j}^{k,n,-}}{\rho_{i,j}^{k,n}-\rho_{i-1,j}^{k,n}  } \right)\\
 &= \left(\frac{\bar{ \lambda}_{x}\partial_{\rho}f_{i-\frac{1}{2}, j}^{k,n}(\zeta_{i-\frac{1}{2},j}^{k,n})+ \bar{\alpha}}{2}\right)\left(\frac{\rho_{i+\frac{1}{2},j}^{k,n, -}-\rho_{i-\frac{1}{2},j}^{k,n,-}}{\rho_{i,j}^{k,n}-\rho_{i-1,j}^{k,n}  } \right)\\
&\le  \frac{ \bar{\lambda}_{x}\lVert \partial_{\rho}f^k\rVert+ \bar{\alpha}}{2} (1+\theta)
\le \frac{1}{3},
\end{nalign}
where $\displaystyle \hat{a}_{i-\frac{1}{2},j}^{k,n} := \left( f^{k,n}_{i-\frac{1}{2}, j}(\rho_{i+\frac{1}{2},j}^{k,n,-})- f^{k,n}_{i-\frac{1}{2}, j}(\rho_{i-\frac{1}{2},j}^{k,n,-}) +\frac{\bar{\alpha}}{ \bar{\lambda}_x}( \rho_{i+\frac{1}{2},j}^{k,n, -}-\rho_{i-\frac{1}{2},j}^{k,n,-})\right)$ and for some  $\zeta^{k,n}_{i-\frac{1}{2},j} \in \mathcal{I}(\rho_{i+\frac{1}{2},j}^{k,n,-},\rho_{i-\frac{1}{2},j}^{k,n,-}).$  

Here, the last inequality follows from the fact that $  \bar{\lambda}_x {\bigl(6(1+\theta)\lVert\partial_{\rho}f^{k} \rVert  +1\bigr)} \leq 4-6\bar{ \alpha}(1+\theta),$ {\cblue {noted from the CFL  condition \eqref{cfl-2D} }}.
Further, hypothesis \hyperlink{H0}{\textbf{(H0)}} together with  the inequality $\bar{\lambda}_x {\bigl(6(1+\theta)\lVert\partial_{\rho}f^k \rVert  +1\bigr)} \leq 6 \bar{\alpha}$ obtained from the CFL condition (\ref{cfl-2D}), yield
\begin{align*}
    a_{i-\frac{1}{2},j}^{k,n} \geq  \frac{-\bar{\lambda}_{x}\lVert \partial_{\rho}f_{i-\frac{1}{2}, j}^{k}\rVert+ \bar{\alpha}}{2}  \left(\frac{\rho_{i+\frac{1}{2},j}^{k,n, -}-\rho_{i-\frac{1}{2},j}^{k,n,-}}{\rho_{i,j}^{k,n}-\rho_{i-1,j}^{k,n}  } \right)\geq 0.
\end{align*}
\par
In a similar way, we obtain the  bound 
\begin{nalign}\label{eq:b_bounds}
    0 \leq b^{k,n}_{i+\frac{1}{2},j} \leq \frac{1}{3}.
\end{nalign}
\par
To estimate the last term of \eqref{eq:RKxsweep}, we use the definition \eqref{LxF-2Dflux} and apply the triangle inequality, leading to
\begin{align*}
  &  \abs[\Big]{F^{k,n}_{i+\frac{1}{2}, j}(\rho_{i+\frac{1}{2},j}^{k,n,-}, \rho_{i-\frac{1}{2},j}^{k,n,+}) - F^{k,n}_{i-\frac{1}{2}, j}(\rho_{i+\frac{1}{2},j}^{k,n,-}, \rho_{i-\frac{1}{2},j}^{k,n,+})} \le J_1+J_2,\\
\end{align*}
where 
\begin{nalign}\label{eq:J1J2}
    J_1&:=\frac{1}{2}\abs[Big]{f^{k}(t^{n}, x_{i+\frac{1}{2}},y_{j},\rho_{i+\frac{1}{2},j}^{k,n,-}, \boldsymbol{A}^{n}_{i+\frac{1}{2}, j})-f^{k}(t^{n}, x_{i-\frac{1}{2}},y_{j},\rho_{i+\frac{1}{2},j}^{k,n,-}, \boldsymbol{A}^{n}_{i-\frac{1}{2}, j})} \,\, \mbox{and}\\
    J_2&:=\frac{1}{2}\abs[big]{f^{k}(t^{n}, x_{i+\frac{1}{2}},y_{j},\rho_{i-\frac{1}{2},j}^{k,n,+}, \boldsymbol{A}^{n}_{i+\frac{1}{2}, j})-f^{k}(t^{n}, x_{i-\frac{1}{2}},y_{j},\rho_{i-\frac{1}{2},j}^{k,n,+}, \boldsymbol{A}^{n}_{i-\frac{1}{2}, j})} .
\end{nalign}
\par
Note that, by the choice of the slope limiter \eqref{eq:slopes}, the face values $\rho_{i-\frac{1}{2},j}^{k,n,+},\\ \;\rho_{i+\frac{1}{2},j}^{k,n,-} \ge 0,\;\, \forall i,j \in \mathbb{Z}.$ Further, as a consequence of Remark \ref{remark:2}, we also have
\begin{nalign}\label{eq:upbd_recon}
    \rho_{i+\frac{1}{2},j}^{k,n,-} &=\rho_{i,j}^{k,n} +\frac{1}{2}(\rho_{i+\frac{1}{2},j}^{k,n,-}-\rho_{i-\frac{1}{2},j}^{k,n,+})\le \rho_{i,j}^{k,n} +\frac{1}{2}|\rho_{i+\frac{1}{2},j}^{k,n,-}-\rho_{i-\frac{1}{2},j}^{k,n,+}| \\ &\le (1+\theta) \rho_{i,j}^{k,n},\\
    \rho_{i-\frac{1}{2},j}^{n,+} &=\rho_{i,j}^{k,n} -\frac{1}{2}(\rho_{i+\frac{1}{2},j}^{k,n,-}-\rho_{i-\frac{1}{2},j}^{k,n,+})\le \rho_{i,j}^{k,n} +\frac{1}{2}|\rho_{i+\frac{1}{2},j}^{k,n,-}-\rho_{i-\frac{1}{2},j}^{k,n,+}| \\ &\le (1+\theta) \rho_{i,j}^{k,n}.
\end{nalign}
\par
Furthermore, by adding and subtracting $f^{k}(t^{n}, x_{i-\frac{1}{2}},y_{j},\rho_{i+\frac{1}{2},j}^{k,n,-}, \boldsymbol{A}^{n}_{i+\frac{1}{2}, j})$ to the term $J_1$ of \eqref{eq:J1J2} and using the hypotheses \hyperlink{H0}{\textbf{(H0)}} and  \hyperlink{H1}{\textbf{(H1)}} together with the expression \eqref{eq:upbd_recon}, we obtain the following estimate 
\begin{nalign}\label{eq:J1bd}
J_1
&\le \frac{1}{2}\left(\abs[big]{f^{k}(t^{n}, x_{i+\frac{1}{2}},y_{j},\rho_{i+\frac{1}{2},j}^{k,n,-}, \boldsymbol{A}^{n}_{i+\frac{1}{2}, j})-f^{k}(t^{n}, x_{i-\frac{1}{2}},y_{j},\rho_{i+\frac{1}{2},j}^{k,n,-}, \boldsymbol{A}^{n}_{i+\frac{1}{2}, j})} \right)\\
& \spc +\frac{1}{2} \left(\abs[big]{f^{k}(t^{n}, x_{i-\frac{1}{2}},y_{j},\rho_{i+\frac{1}{2},j}^{k,n,-}, \boldsymbol{A}^{n}_{i+\frac{1}{2}, j})}\right) \\ & \spc +\frac{1}{2}\left(\abs[big]{f^{k}(t^{n}, x_{i-\frac{1}{2}},y_{j},\rho_{i+\frac{1}{2},j}^{k,n,-}, \boldsymbol{A}^{n}_{i-\frac{1}{2}, j})}\right) \\
&=\frac{1}{2} \left(|\partial_x f^{k}(t^{n}, \bar{x_{i}},y_{j},\rho_{i+\frac{1}{2},j}^{k,n,-}, \boldsymbol{A}^{n}_{i+\frac{1}{2}, j})| \Delta x \right)
\\& \spc +  \frac{1}{2}\left(|\partial_\rho f^{k}(t^{n}, x_{i-\frac{1}{2}},y_{j},\bar{\rho_i}, \boldsymbol{A}^{n}_{i+\frac{1}{2}, j})|\rho_{i+\frac{1}{2},j}^{k,n,-} \right) \\
&\spc + \frac{1}{2}\left( |\partial_\rho f^{k}(t^{n}, x_{i-\frac{1}{2}},y_{j},\hat{\rho_i}, \boldsymbol{A}^{n}_{i-\frac{1}{2}, j})|\rho_{i+\frac{1}{2},j}^{k,n,-}\right)
\\
&\le \left(\norm{\partial_\rho f^k} +\frac{1}{2} M \Delta x\right)\rho_{i+\frac{1}{2},j}^{k,n,-} \le \left(\norm{\partial_\rho f^k} +\frac{1}{2} M \Delta x\right)(1+\theta)\rho_{i,j}^{k,n}.
\end{nalign}
where $\bar{x}_{i} \in (x_{\imh},x_{\iph})$ and \cblue{$\bar{\rho}_{i}, \hat{\rho}_{i} \in \mathcal{I}(0,\rho_{i+\frac{1}{2},j}^{k,n,-}).$} The term $J_2$ is treated similarly, to obtain
\begin{nalign}\label{eq:J2bd}
    J_2 \le (\norm{\partial_\rho f^k} +\frac{1}{2} M \Delta x)\rho_{i-\frac{1}{2},j}^{n,+} \le (\norm{\partial_\rho f^k} +\frac{1}{2} M \Delta x)(1+\theta)\rho_{i,j}^{k,n}.
\end{nalign}
\par
Combining the estimates \eqref{eq:J1bd} and \eqref{eq:J2bd}, we get
\begin{nalign}\label{eq:J1+J2}
    J_1+J_2  \le \left(2(1+\theta)\norm{\partial_\rho f^k} + M \Delta x\right)\rho_{i,j}^{k,n}.
\end{nalign}
\par
Next, in view of \eqref{eq:J1+J2}, we arrive at the estimate
\begin{nalign}\label{eq:Fdif_2ndbd}
    &\bar{ \lambda_x}\abs[\Big]{F^{k,n}_{i+\frac{1}{2}, j}(\rho_{i+\frac{1}{2},j}^{k,n,-}, \rho_{i-\frac{1}{2},j}^{k,n,+}) - F^{k,n}_{i-\frac{1}{2}, j}(\rho_{i+\frac{1}{2},j}^{k,n,-}, \rho_{i-\frac{1}{2},j}^{k,n,+})} \\& \le \bar{\lambda}_x\left(2(1+\theta)\norm{\partial_\rho f^k} + M \Delta x\right)\rho_{i,j}^{k,n}\le \frac{1}{3}\rho_{i,j}^{k,n},
\end{nalign}
where we use the conditions $\bar{\lambda}_x(6(1+\theta)\norm{\partial_\rho f^k} + 1) \le 1$ (derived from \eqref{cfl-2D}) and $\displaystyle \Delta x\leq \frac{1}{3M}.$

Thus, we derive the following estimate using the expressions \eqref{eq:a_upperbd}, \eqref{eq:b_bounds} and  \eqref{eq:Fdif_2ndbd} in (\ref{eq:RKxsweep})
\begin{align}
V_{i,j}^{k,(1)}& \geq \Bigl(1-a^{k,n}_{i-\frac{1}{2},j}- b^{k,n}_{i+\frac{1}{2},j} \Big)\rho_{i, j}^{k,n}+a^{k,n}_{i-\frac{1}{2},j}\rho_{i-1, j}^{k,n} +b^{k,n}_{i+\frac{1}{2},j} \rho_{i+1, j}^{k,n} \no\\
    &\spc- \bar{\lambda}_{x}\abs[\Big]{F^{k,n}_{i+\frac{1}{2}, j}(\rho_{i+\frac{1}{2},j}^{k,n,-}, \rho_{i-\frac{1}{2},j}^{k,n,+}) - F^{k,n}_{i-\frac{1}{2}, j}(\rho_{i+\frac{1}{2},j}^{k, n,-}, \rho_{i-\frac{1}{2},j}^{k,n,+})} \no \\
    &\ge  \Bigl(1-a^{k,n}_{i-\frac{1}{2},j}- b^{k,n}_{i+\frac{1}{2},j} -\frac{1}{3}\Big)\rho_{i, j}^{k,n}+a^{k,n}_{i-\frac{1}{2},j}\rho_{i-1, j}^{k,n} +b^{k,n}_{i+\frac{1}{2},j} \rho_{i+1, j}^{k,n} \ge 0 \no.
\end{align}
\par
    In a similar way, one can show that $W_{ij}^{k,(1) } \ge 0$ and finally we deduce that $\displaystyle\rho^{k,(1)}_{ij} \ge 0,$ for all $i,j \in \mathbb{Z}.$ {\cblue {Eventually, we obtain $\displaystyle\rho^{k,(2)}_{ij} \ge 0,$ for all $i,j \in \mathbb{Z},$ analogously. }}
Thus, \cblue{ by considering \eqref{RK-2d} }, we conclude that the final numerical solutions  satisfy $\displaystyle \rho^{k,n+1}_{ij} \ge 0,$ for $i,j \in \mathbb{Z},$ thereby completing  the proof.
\end{proof}
\cblue{
\begin{remark}
    When $\theta=0$, the second-order scheme \eqref{RK-2d} reduces to a first-order in space and second-order in time scheme (see equation (5.2) in \cite{gowda2023}), and the CFL conditions (\ref{cfl-2D}) reduce to
\begin{align}\label{eq:CFLfo}
\bar{\lambda}_x \leq \frac{\min\{1, 4- 6\bar{\alpha} , 6\bar{\alpha}\}}{(6\lVert\partial_{\rho}f^k \rVert  +1)} , \; \bar{\lambda}_y\le \frac{\min\{1, 4- 6\bar{\beta} , 6\bar{\beta}\}}{(6\lVert\partial_{\rho}g^k \rVert  +1)}. 
 \end{align}
 \textcolor{black}{Moreover, we observe that the Euler forward step \eqref{rk-2d-1stEulerstep} in the second-order scheme reduces to the first-order scheme \eqref{eq:foscheme} when $\theta = 0$. Consequently, by setting $\theta = 0$ in  the proof of Theorem \ref{thm:positivity}, we obtain that the FO scheme is positivity preserving under the CFL condition \eqref{eq:CFLfo}.}
\end{remark}}
\cblue{
\begin{remark}\label{rem:alphabeta}
    For the first-order scheme \eqref{eq:foscheme}, the CFL condition \eqref{eq:CFLfo} implies that the coefficients $\alpha, \beta$ in the numerical flux \eqref{LxF-2Dflux} should satisfy $\alpha, \beta \in \left(0, \frac{1}{3}\right)$. Similarly, for the second-order scheme \eqref{RK-2d}, the CFL condition \eqref{cfl-2D} requires that $\alpha, \beta \in \left(0, \frac{2}{9}\right)$.
\end{remark}}

We now present a corollary to Theorem \ref{thm:positivity}, which will aid in proving the $\mathrm{L}^\infty$-stability.
\begin{corollary}($\mathrm{L}^{1}$-  stability)\label{corollary:L1stable} Under the assumptions of Theorem \ref{thm:positivity}, for a non-negative initial data  $\boldsymbol{\rho}_{0} \in \mathrm{L}^{1} \cap \mathrm{L}^{\infty}(\mathbb{R}^{2}; \mathbb{R}_{+}^{N}),$ the approximate solution $\boldsymbol{\rho}_{\Delta}$ obtained from the scheme \eqref{RK-2d} satisfies \begin{align}\label{eq: L1stability}
 \norm{\rho^{k}_{\Delta}(t)}_{\mathrm{L}^1} =   \norm{\rho^{k}_\Delta(0)}_{\mathrm{L}^1}, 
\end{align}
for all $k \in \{1,2, \dots, N\}$
and $t \in \mathbb{R}_+.$
\begin{proof}
By Theorem \eqref{thm:positivity}, the assumptions  $\rho^{k,0}_{i,j} \ge 0,$ imply that $\rho^{k,n}_{i,j} \ge 0$ for all $i,j \in \mathbb{Z}$ and $n \in \mathbb{N}.$ Additionally, each stage in the Runge-Kutta time stepping satisfies  $\rho^{k,(1)}_{i,j},\rho^{k,(2)}_{i,j} \ge 0$ for all $i,j \in \mathbb{Z}.$  
Therefore, we have the following
\begin{align*}
    \norm{\rho^{k,(1)}}_{\mathrm{L}^1} &= \Delta x  \Delta y  \sum_{i,j \in \mathbb{Z}}\rho^{k,(1)}_{i,j} = \Delta x  \Delta y  \sum_{i,j \in \mathbb{Z}}\rho^{k,n}_{i,j}
    \end{align*} and 
    \begin{align*}
    \norm{\rho^{k,(2)}}_{\mathrm{L}^1} &= \Delta x  \Delta y  \sum_{i,j \in \mathbb{Z}}\rho^{k,(2)}_{i,j} = \Delta x  \Delta y  \sum_{i,j \in \mathbb{Z}}\rho^{k,(1)}_{i,j}. 
\end{align*}
Consequently, we obtain the result
\begin{align*}
    \norm{\rho^{k,n+1}}_{\mathrm{L}^1}= \Delta x \Delta y  \sum_{i,j} \rho_{i,j}^{k,n+1} &= \Delta x  \Delta y  \sum_{i,j \in \mathbb{Z}}  \frac{\rho_{i,j}^{k,(2)}+\rho_{i,j}^{k,n}}{2}\\&= \Delta x  \Delta y  \sum_{i,j \in \mathbb{Z}} \rho^{k,n}_{i,j}  = \norm{\rho^{k,n}}_{\mathrm{L}^1}.
\end{align*}
The equality \eqref{eq: L1stability} now follows immediately.
\end{proof}
\end{corollary}

\section{\texorpdfstring{$\mathrm{L}^\infty$ stability}{L-infinity stability}}\label{section:Linfinty}
In this section, we establish that the second-order scheme given by  \eqref{RK-2d} exhibits $\mathrm{L}^\infty$- stability. 

\begin{theorem}($\mathrm{L}^\infty$-stability) Let $\boldsymbol{\rho}_{0} \in \mathrm{L}^{1} \cap \mathrm{L}^{\infty}(\mathbb{R}^{2}; \mathbb{R}_{+}^{N})$. If the hypotheses \hyperlink{H0}{\textbf{(H0)}}, \hyperlink{H1}{\textbf{(H1)}} and \hyperlink{H2}{\textbf{(H2)}} and the CFL condition \eqref{cfl-2D} hold along with the mesh-size restriction  $\displaystyle \Delta x, \Delta y \leq \frac{1}{3M}$, then there exists a constant $C \geq 0$ depending only on $\boldsymbol{\rho}_{0}, \boldsymbol{\eta}, \boldsymbol{\nu}, \{f^k\}_{k=1}^{N} $ and $ \{g^k\}_{k=1}^{N}$ such that  the approximate solution $\boldsymbol{\rho}_{\Delta}$ obtained from  the second-order scheme (\ref{RK-2d}) satisfies 
\begin{align*}
\norm{\boldsymbol{\rho}_\Delta(t)} \le \norm{\boldsymbol{\rho}_\Delta(0)} e^{Ct},
\end{align*} for any $t \in \mathbb{R}_+.$
\begin{proof}
By Corollary \ref{corollary:L1stable} and applying the mean value theorem, we observe that the discrete convolutions \eqref{eq:A} satisfy the following estimate
\begin{nalign}\label{eq:Adifbd}
&\norm{\boldsymbol{A}^{n}_{i+\frac{1}{2}, j}-\boldsymbol{A}^{n}_{i-\frac{1}{2}, j}} \\ & \spc = \norm[\Bigg]{\left(\Delta x \Delta y \sum_{k=1}^{N}\sum_{p,l \in \mathbb{Z}}\left(\eta_{i+\frac{1}{2}-l, j-p}^{q,k}-\eta_{i-\frac{1}{2}-l, j-p}^{q,k}\right)        \ \rho^{k,n}_{l, p}\right)_{q=1}^{m}} \\& \spc\leq    \Delta x\left(\norm{\partial_x \boldsymbol{\eta}} \norm{\boldsymbol{\rho}_\Delta(t^n)}_{\mathrm{L}^1}\right)\\
& \spc \leq \Delta x\left(\norm{\partial_x \boldsymbol{\eta}} \norm{\boldsymbol{\rho}_\Delta(0)}_{\mathrm{L}^1}\right).\end{nalign}
  Further, invoking  the estimate \eqref{eq:Adifbd} and using the hypotheses \hyperlink{H0}{\textbf{(H0)}} and \hyperlink{H1}{\textbf{(H1)}},  we arrive at the following estimate  
\begin{nalign}
    &\abs[\Big]{F^{k,n}_{i+\frac{1}{2}, j}(\rho_{i+\frac{1}{2},j}^{k,n,-}, \rho_{i-\frac{1}{2}, j}^{k,n,+}) - F^{k,n}_{i-\frac{1}{2}, j}(\rho_{i+\frac{1}{2},j}^{k,n,-}, \rho_{i-\frac{1}{2},j}^{k,n,+})} \\
   &\leq \frac{1}{2} \abs[\Big]{ \left(f^{k}(t^{n}, x_{i+\frac{1}{2}},y_{j},\rho_{i+\frac{1}{2},j}^{k,n,-}, \boldsymbol{A}^{n}_{i+\frac{1}{2}, j})- f^{k}(t^{n}, x_{i-\frac{1}{2}},y_{j},\rho_{i+\frac{1}{2},j}^{k,n,-}, \boldsymbol{A}^{n}_{i-\frac{1}{2}, j}) \right)}\\
   & \spc+  \frac{1}{2}  \abs[\Big]{\left(f^{k}(t^{n}, x_{i+\frac{1}{2}},y_{j},\rho_{i-\frac{1}{2},j}^{k,n,+}, \boldsymbol{A}^{n}_{i+\frac{1}{2}, j})- f^{k}(t^{n}, x_{i-\frac{1}{2}},y_{j},\rho_{i-\frac{1}{2},j}^{k,n,+}, \boldsymbol{A}^{n}_{i-\frac{1}{2}, j}) \right)}  \\
&\leq \frac{1}{2}|\partial_{x}f^{k}(t^n,\bar{x_i},y_j,\rho_{i+\frac{1}{2},j}^{k,n,-},\bar{\boldsymbol{A}}^{n}_{i, j})\Delta x| \\ & \spc+\frac{1}{2} \left(\norm{\nabla_A f^{k}(t^n,\bar{x_i},y_j,\rho_{i+\frac{1}{2},j}^{k,n,-},\bar{\boldsymbol{A}}_{i, j}^{n})}\norm{\boldsymbol{A}^{n}_{i+\frac{1}{2}, j}-\boldsymbol{A}^{n}_{i-\frac{1}{2}, j}}  \right) \\
&\spc + \frac{1}{2} |\partial _x f^{k}(t^n, \tilde{x}_{i},y_j,\rho_{i-\frac{1}{2},j}^{k,n,+},\tilde{\boldsymbol{A}}^{n}_{i, j})\Delta x|\\ &\spc + \frac{1}{2}\left(\norm{\nabla_A f^k(t^n,\tilde{x_i},y_j,\rho_{i-\frac{1}{2},j}^{k,n,+},\tilde{\boldsymbol{A}}^{n}_{i, j})} \norm{\boldsymbol{A}^{n}_{i+\frac{1}{2}, j}-\boldsymbol{A}^{n}_{i-\frac{1}{2}, j}} \right) \\
&\leq \frac{1}{2} \left(M \rho_{i+\frac{1}{2},j}^{k,n,-} \Delta x\Bigl(\norm{\partial_x \boldsymbol{\eta}} \norm{\boldsymbol{\rho}_\Delta(0)}_{\mathrm{L}^1}+1\Bigr) \right)\\ & \spc +\frac{1}{2} \left(M \rho_{i-\frac{1}{2},j}^{k,n,+} \Delta x\Bigl(\norm{\partial_x \boldsymbol{\eta}} \norm{\boldsymbol{\rho}_\Delta(0)}_{\mathrm{L}^1}+1\Bigr) \right) \\
&= M \rho_{i,j}^{k,n} \Delta x\left(\norm{\partial_x \boldsymbol{\eta}}\norm{\boldsymbol{\rho}_\Delta(0)}_{\mathrm{L}^1}+1\right),\label{bound1}
\end{nalign}
\cblue{where $\bar{x}_{i}, \tilde{x}_i \in (x_{\imh}, x_{\iph})$ and $\bar{\boldsymbol{A}}^{n}_{i, j}, \tilde{\boldsymbol{A}}^{n}_{i, j} \in \mathcal{I}(\boldsymbol{A}^{n}_{i+\frac{1}{2}, j}, \boldsymbol{A}^{n}_{i-\frac{1}{2}, j}).$ }
Now, in view of the estimates \eqref{eq:a_upperbd}, \eqref{eq:b_bounds} and \eqref{bound1}, the terms $V_{ij}^{k,(1)}$ in  \eqref{eq:RKxsweep} can be bounded as
\begin{nalign}\label{eq:V_bd_max}
| V_{ij}^{k,(1)} | & \le \Bigl(1-a^{k,n}_{i-\frac{1}{2},j}- b^{k,n}_{i+\frac{1}{2},j} \Big)\abs{\rho_{i, j}^{k,n}}+a^{k,n}_{i-\frac{1}{2},j}\abs{\rho_{i-1, j}^{k,n}} +b^{k,n}_{i+\frac{1}{2},j} \abs{\rho_{i+1, j}^{k,n}} \\
    &\hspace{0.5cm}+ \bar{\lambda}_{x}\abs[\Big]{F^{k,n}_{i+\frac{1}{2}, j}(\rho_{i+\frac{1}{2},j}^{k,n,-}, \rho_{i-\frac{1}{2},j}^{k,n,+}) - F^{k,n}_{i-\frac{1}{2}, j}(\rho_{i+\frac{1}{2},j}^{k,n,-}, \rho_{i-\frac{1}{2},j}^{k,n,+})}    \\
    &\le \norm{\rho^{k}_\Delta(t^n)}\Bigl(1+2M  \Delta t \Bigl(\norm{\partial_x \boldsymbol{\eta}} \norm{\boldsymbol{\rho}_\Delta(0)}_{\mathrm{L}^1}+1\Bigr) \Bigr). 
\end{nalign}
An analogous argument for $W_{ij}^{k,(1)}$ yields
\begin{nalign}\label{eq:w_bd_max}
    | W_{ij}^{k,(1)} |\leq \norm{\rho^{k}_\Delta(t^{n})}\Bigl(1+2M  \Delta t\Bigl(\norm{\partial_y \boldsymbol{\nu}} \norm{\boldsymbol{\rho}_\Delta(0)}_{\mathrm{L}^1}+1\Bigr) \Bigr).
\end{nalign}

Therefore, using the bounds \eqref{eq:V_bd_max} and \eqref{eq:w_bd_max} in \eqref{eq:euler_as_v+w}, it follows that
\begin{align}
| \rho^{k,(1)}_{ij} | &\le  \norm{\rho^k_\Delta(t^{n})} \Bigl(1+2M  \Delta t \Bigl(\max \{\norm{\partial_x \boldsymbol{\eta}},   \norm{\partial_y \boldsymbol{\nu}} \}  \norm{\boldsymbol{\rho}_\Delta(0)}_{\mathrm{L}^1}+1\Bigr) \Bigr).\no 
\end{align}
\par
Similar arguments for the second forward Euler step  \eqref{rk-2d-2stEulerstep} give us the estimate
\begin{nalign}\label{eq:rho2_bd}
|\rho^{k,(2)}_{ij} | &\le  \norm{\rho^{k,(1)}_\Delta} \Bigl(1+2M  \Delta t\Bigl(\max \{\norm{\partial_x \boldsymbol{\eta}},   \norm{\partial_y \boldsymbol{\nu}} \}  \norm{\boldsymbol{\rho}_\Delta(0)}_{\mathrm{L}^1}+1\Bigr) \Bigr)  \\
&\le \norm{\rho^k_\Delta(t^n)} \Bigl(1+2M  \Delta t\Bigl(\max \{\norm{\partial_x \boldsymbol{\eta}},   \norm{\partial_y\boldsymbol{\nu}} \} \norm{\boldsymbol{\rho}_\Delta(0)}_{\mathrm{L}^1}+1\Bigr) \Bigr)^2. 
\end{nalign}
\par
Finally, in light of the estimate \eqref{eq:rho2_bd}, we deduce that 
\begin{nalign}\label{eq:L_infbd}
\abs{\rho_{ij}^{k,n+1}}&=\frac{1}{2}(\abs{\rho_{ij}^{k,n}}+\abs{\rho_{ij}^{k,(2)}})\\&\le \norm{\rho^k_\Delta(t^n)} \Bigl(1+2M  \Delta t\Bigl(\max \{\norm{\partial_x \boldsymbol{\eta}},   \norm{\partial_y \boldsymbol{\nu}}\} \norm{\boldsymbol{\rho}_\Delta(0)}_{\mathrm{L}^1}+1\Bigr) \Bigr)^2 \\
 &\le  \norm{\rho^k_\Delta(t^{n-1})} \Bigl(1+2M  \Delta t\Bigl(\max \{\norm{\partial_x \boldsymbol{\eta}},   \norm{\partial_y \boldsymbol{\nu}} \}  \norm{\boldsymbol{\rho}_\Delta(0)}_{\mathrm{L}^1}+1\Bigr) \Bigr)^4 \\
&\;\;\;\;\vdots \\
& \le \norm{\rho^k_\Delta(0)}\Bigl(1+2M  \Delta t\Bigl(\max \{\norm{\partial_x \boldsymbol{\eta}},   \norm{\partial_y \boldsymbol{\nu}}\}\norm{\boldsymbol{\rho}_\Delta(0)}_{\mathrm{L}^1}+1\Bigr) \Bigr)^{2(n+1)} \\
& \le \norm{\boldsymbol{\rho}_\Delta(0)} e^{Ct},
\end{nalign}
\cblue{for $t = (n+1)\Delta t,$} where
$\displaystyle C:=4M\Bigl(1+\max \{\norm{\partial_x \boldsymbol{\eta}},   \norm{\partial_y \boldsymbol{\nu}} \}\norm{\boldsymbol{\rho}_\Delta(0)}_{\mathrm{L}^1}\Bigr).$ The estimate \eqref{eq:L_infbd} completes the proof.
\end{proof}
\end{theorem}

\section{Numerical experiments}\label{section:numericalexp}
This section presents a performance comparison of \cblue{the} first- and second-order schemes for two-dimensional non-local conservation laws. \cblue{We primarily consider two types of test cases: one involving a scalar equation and the other a system, both adhering to the framework of \eqref{eq:system}.} In all the numerical results, the time-step $\Delta t$ is computed using the CFL condition \eqref{cfl-2D} corresponding to the second-order scheme  and the computational domain $[x_1,x_2] \times [y_1, y_2]$ is discretized into $(n_x\times n_y)$ number of Cartesian cells, where the grid sizes $\displaystyle \Delta x = ({x_2 - x_1})/n_x$ and $\displaystyle \Delta y = ({y_2 - y_1})/n_y.$ \cblue{The coefficient in \eqref{eq:slopes} is set to be $\theta = 0.5,$ and the parameters in the numerical fluxes \eqref{LxF-2Dflux} are both chosen to be $\alpha = \beta =1/6,$ in all the examples.}   The initial and boundary conditions  are prescribed in the description of each example. Hereafter, the first-order scheme \eqref{eq:foscheme} and the second-order scheme \eqref{RK-2d} will be referred to as FO and SO, respectively.
\par
\begin{example}\label{ex-cm1} In this example, we consider the two-dimensional macroscopic crowd dynamics problem studied in \cite{aggarwal2015}, where the density $\rho$ of pedestrians  is modeled to evolve according to the scalar non-local conservation law:
\begin{nalign}\label{crowdproblem}
    \partial_{t}\rho + \nabla \cdot (\rho(1-\rho)(1-\rho*\mu) \vec{\bs v}) = 0.
\end{nalign}
Here, the convolution  is given by
\begin{align*}
\rho*\mu(t,x,y) = \int\int_{\mathbb{R}^{2}} \mu(x- x^{\prime}, y- y^{\prime}) \rho(t, x^{\prime}, y^{\prime})\dif x^{\prime} \dif y^{\prime}.
\end{align*}
\par \cblue{The smooth kernel function $\mu$ quantifies the weight assigned by pedestrians to their surrounding crowd density}, while the vector field $\vec{\boldsymbol{v}}(x,y) = (v^{1}(x, y), v^{2}(x, y))^{T}$ describes the path they follow. It is evident that \eqref{crowdproblem} aligns with  the framework of \eqref{eq:system} (see Lemma 3.1 in \cite{aggarwal2015}).
{\cblue{We examine a scenario where two groups of individuals start from two different locations within the domain $[0, 10] \times [-1, 1],$ move in the same direction  and eventually stop at the spot $\{9.5\}\times [-1,1].$ To account for this dynamics, the velocity vector field is chosen as}}
\begin{align*}
\vec{\boldsymbol{v}}(x,y) = \begin{bmatrix}
           (1- y^{2})^{3} \textrm{exp}(-1/(x-9.5)^{2}) \chi_{(-\infty, 9.5]\times [-1, 1]}(x, y) \\
           -2y \textrm{exp}(1 - 1/y^{2})\\
         \end{bmatrix},
\end{align*}
\cblue{where for $\Omega \subseteq \mathbb{R}^2,$ $\chi_{\Omega}$ denotes the indicator function of $\Omega.$} Further, the kernel function is defined to be of compact support in a disk of radius $r=0.4$ centered at the origin:  \begin{nalign}\label{eq:kern_crowd}
\mu(x,y) = \frac{1}{\int\int_{\mathbb{R}^{2}} \tilde{\mu}}\tilde{\mu}(x, y),
\end{nalign}
where
\begin{align*}
        \tilde{\mu}(x, y) = (0.16 - x^{2} - y^{2})^{3} \chi_{\{(x,y): x^{2} + y^{2} \leq 0.16\}}(x, y).
\end{align*}
\par
Note that the kernel function $\mu$ in \eqref{eq:kern_crowd} attains a global maximum at the origin $(0,0)$ and decreases radially, reflecting the idea that pedestrians prioritize nearby crowd density over distant ones.
We  solve the problem \eqref{crowdproblem} with  the initial datum:
\begin{align}\label{eq:ic}
\rho^{0}(x, y) = \chi_{ [1, 4] \times[0.1, 0.8]}(x, y) + \chi_{[2, 5] \times [-0.8, -0.1]}(x, y),
\end{align}
given in Fig.\ref{fig:ic_crowd}, {\cblue { along with \textquoteleft no flow\textquoteright \,   boundary conditions}  } on all the boundaries of the domain.
Throughout this example, {\cblue {based on the CFL condition \eqref{cfl-2D} for the SO scheme, we set a common time-step for both the FO and SO schemes, $\Delta t = 0.026 \Delta x.$}} \cblue{This time-step is computed from \eqref{cfl-2D} by using the fact that $\norm{\partial_{\rho}f},\norm{\partial_{\rho}g}\leq 2,$ in this example.} The numerical solutions are computed in the domain $[0,10]\times[-1,1].$
First, we compute the solution at time $T=4.0,$ for both the FO and SO schemes and show that the FO scheme solutions converge towards the SO scheme solutions as the mesh is refined. This is explained in Fig. \ref{fig:crowd_foso},  where Fig. \ref{fig:crowd_foso} (a), (b) and (c) display the solutions obtained from the FO scheme, while (d) corresponds to the SO scheme.  The results clearly show that the mesh size for the FO scheme must be refined at least four times to obtain a solution profile comparable to that of the SO scheme, highlighting the  importance of the SO scheme.

In Fig. \ref{fig:crowd}, we display the numerical solutions at various time levels, $T \in \{8.0, 12.0, 16.0, 20.0\},$ computed using  both the FO and SO schemes with the same initial data as in \eqref{eq:ic}. By comparing the solution profiles, we observe significant differences between the solutions obtained from the FO and SO schemes. Additionally, we note that solutions generated using the SO scheme remain positive and exhibit $\mathrm{L}^{\infty}$-stability, thus confirming the theoretical results.  
\end{example}
\begin{example}\label{ex-cm2} We compute the \cblue{experimental order of convergence} (E.O.C.) for both the FO and SO schemes using the problem \eqref{crowdproblem} and initial condition \eqref{eq:ic} presented in Example \ref{ex-cm1}, and compare their performance. For a uniform grid with $\Delta x = \Delta y$, we denote $h := \Delta x$. Since the exact solution for the problem \eqref{crowdproblem} with the initial condition \eqref{eq:ic} is unavailable, the E.O.C. is calculated based on the $\mathrm{L}^1$-error between solutions obtained for mesh sizes $h$, $h/2$, and so on. The E.O.C. is determined using the formula:
$$\gamma := \log \left ( \frac{\|{\rho}_h - {\rho}_{\frac{h}{2}}\|_{{\mathrm{L}}^1}}{\|{\rho}_{\frac{h}{2}} - {\rho}_{\frac{h}{4}}\|_{{\mathrm{L}}^1}} \right ) / \log 2.$$
\par
Here, the numerical solutions are computed up to  time $T =0.2$ 
for mesh size $h \in \{ 0.05, 0.025, 0.00125, 0.00625, 0.003125\}$ in the computational domain $[0,10]\times[-1,1].$  Both the FO and SO scheme solutions are computed with the same time step \cblue{$\Delta t = 0.026 \Delta x.$} The results summarized in  Table \ref{table:eoa} indicate that the FO scheme achieves an E.O.C. of $\gamma \approx 0.5,$ while the SO scheme reaches an E.O.C. of $\gamma\approx 0.8.$ 

\begin{figure}
     \centering
         
         \centering
         \includegraphics[scale = 0.35 ]{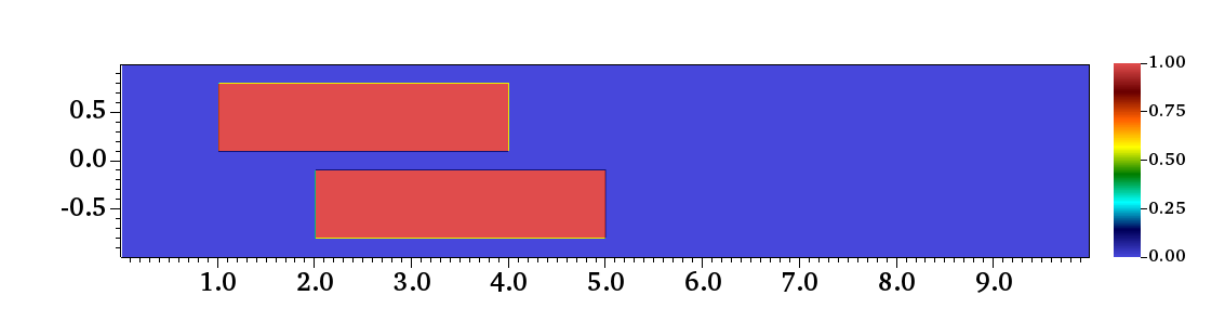}
         \caption{Example \ref{ex-cm1}: Initial condition $\rho^0$ for the problem \eqref{crowdproblem} computed on a mesh of size $\Delta x = \Delta y = 0.0125.$} 
         \label{fig:ic_crowd}
\end{figure}
\begin{figure}
     \centering
          \begin{subfigure}[b]{0.48\textwidth}
         \centering
         \includegraphics[width=\textwidth]{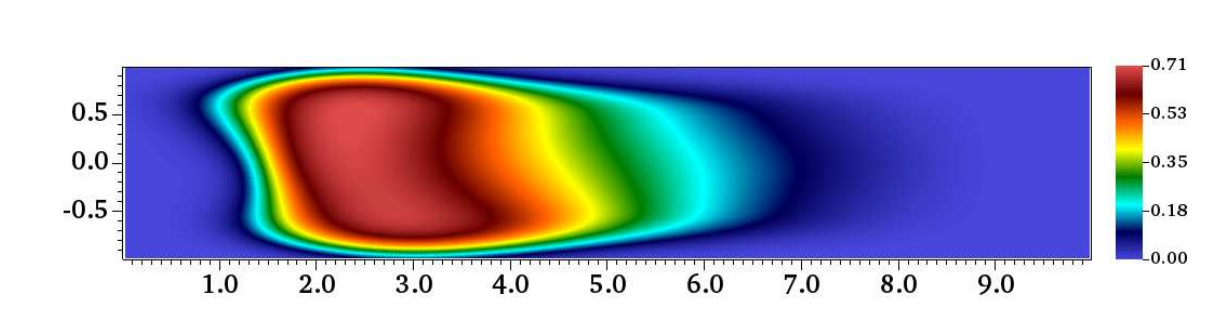}
         \caption{FO-scheme, $\Delta x = \Delta y = 0.025 (400\times \\80)$} 
     \end{subfigure}\hspace{0.1cm}
     \begin{subfigure}[b]{0.48\textwidth}
         \centering
         \includegraphics[width=\textwidth]{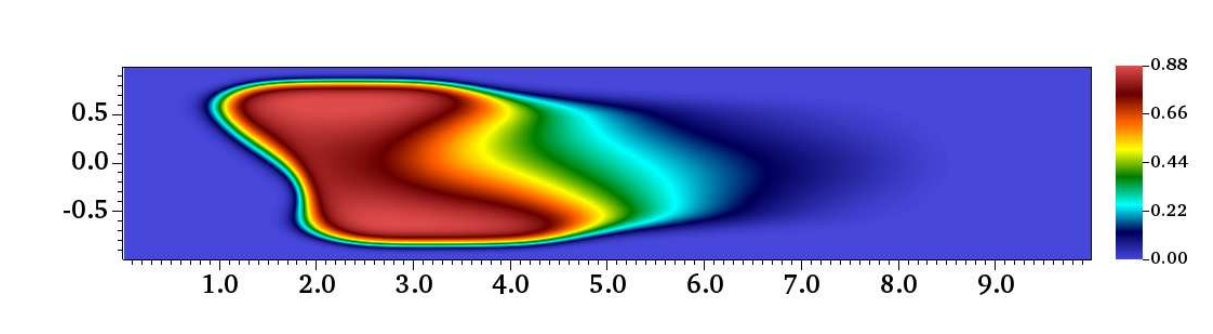}
         \caption{FO-scheme, $\Delta x = \Delta y = 0.0125\,(800 \times 160)$} 
     \end{subfigure}\\
          \begin{subfigure}[b]{0.48\textwidth}
          \centering
         \includegraphics[width=\textwidth]{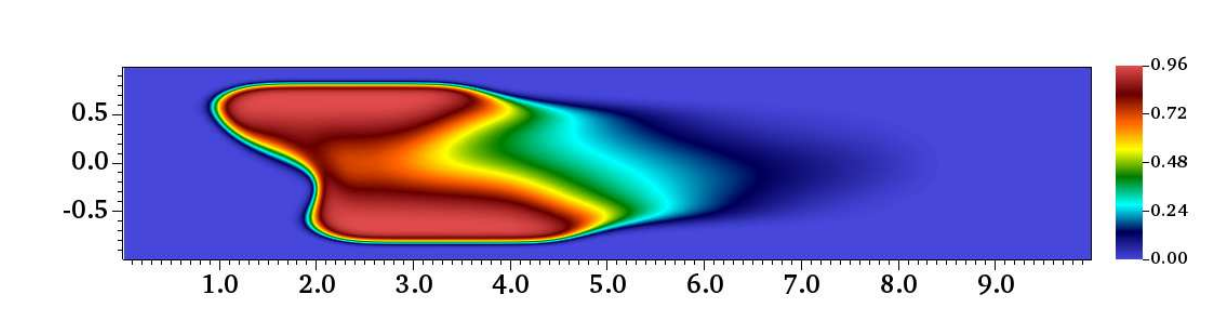}
         \caption{FO-scheme, $\Delta x=\Delta y=0.00625$\, $( 1600\times 320) $} 
     \end{subfigure} \hspace{0.1cm}
     \begin{subfigure}[b]{0.48\textwidth}
         \centering
         \includegraphics[width=\textwidth]{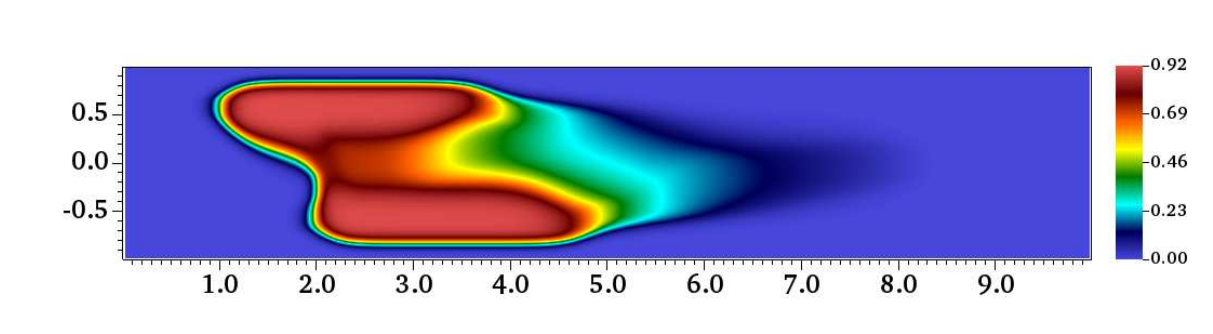}
         \caption{SO-scheme, $\Delta x = \Delta y = 0.025\, (400 \times\\ 80) $ } 
     \end{subfigure}
 \caption{Example \ref{ex-cm1}: Numerical solution $\rho$ of the problem \eqref{crowdproblem} with initial data \eqref{fig:ic_crowd}, obtained using the FO scheme with a mesh of resolution (a) ($800 \times 160$), (b) ($1600 \times 320$), (c) ($3200 \times 640$) and the SO scheme with a resolution of (d) $(800 \times 160)$ at time $t=4.0.$ In all the plots, the time step is taken as \cblue{$\Delta t = 0.026 \Delta x.$}}
 \label{fig:crowd_foso}
 \end{figure}

\begin{figure}
     \centering
          \begin{subfigure}[b]{0.48\textwidth}
         \centering
         \includegraphics[width=\textwidth]{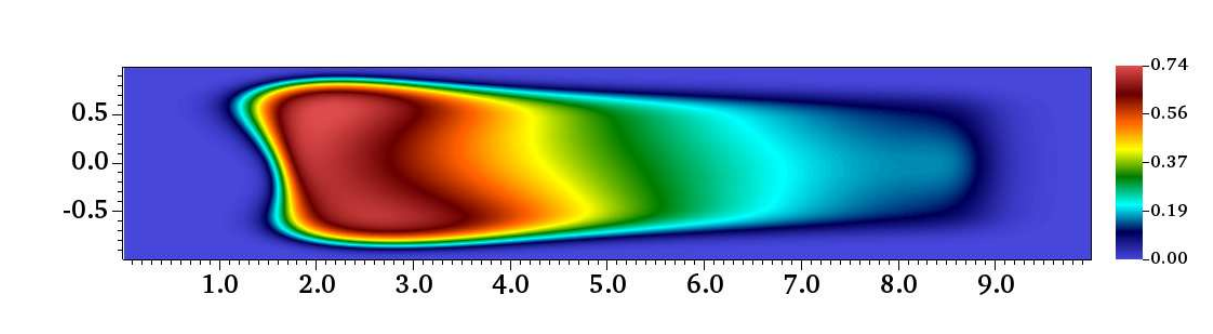}
         \caption{$\rho$ at time $t =8,$ FO-scheme, $\Delta x = \Delta y = 0.0125$} 
     \end{subfigure}\hspace{0.3cm}
      \begin{subfigure}[b]{0.48\textwidth}
         \centering
         \includegraphics[width=\textwidth]{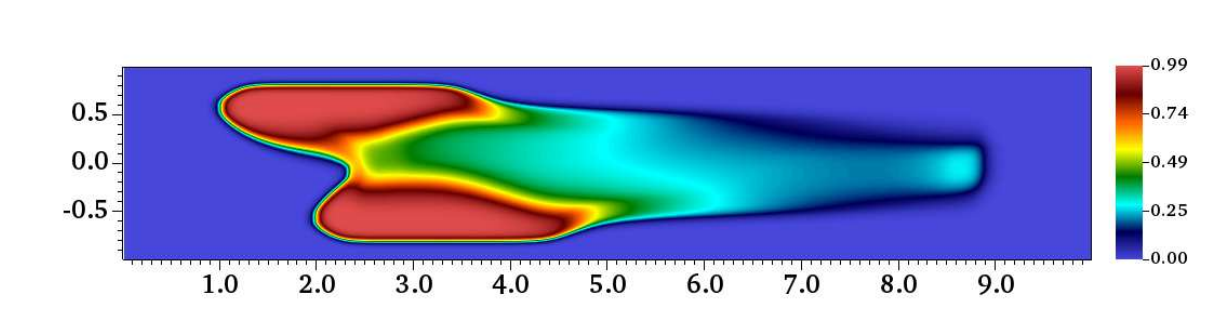}
         \caption{$\rho$ at time $t =8,$ SO-scheme, $\Delta x = \Delta y = 0.0125$} 
     \end{subfigure}\\

          \begin{subfigure}[b]{0.48\textwidth}
         \centering
         \includegraphics[width=\textwidth]{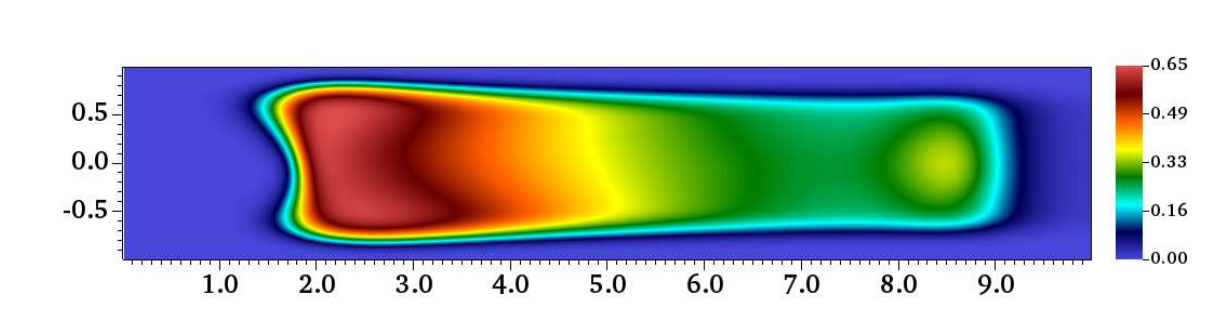}
         \caption{$\rho$ at time $t =12,$ FO-scheme, $\Delta x = \Delta y = 0.0125$} 
     \end{subfigure}\hspace{0.3cm}
       \begin{subfigure}[b]{0.48\textwidth}
         \centering
         \includegraphics[width=\textwidth]{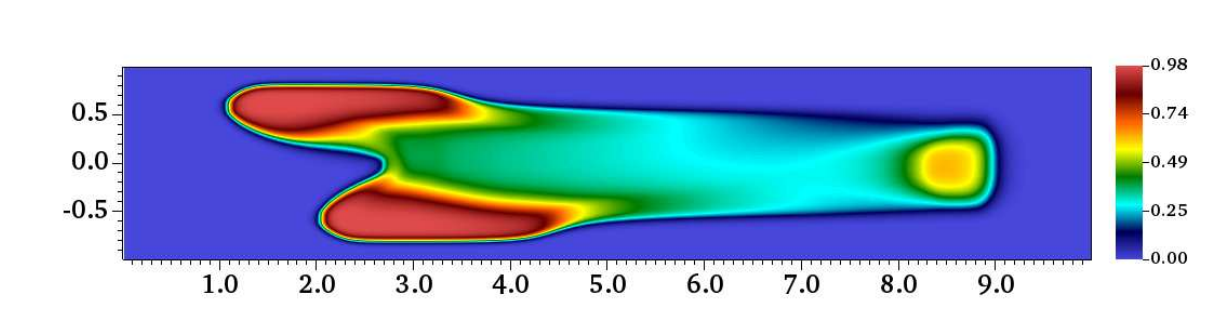}
         \caption{$\rho$ at time $t =12,$ SO-scheme, $\Delta x = \Delta y = 0.0125$} 
     \end{subfigure}\\

               \begin{subfigure}[b]{0.48\textwidth}
         \centering
         \includegraphics[width=\textwidth]{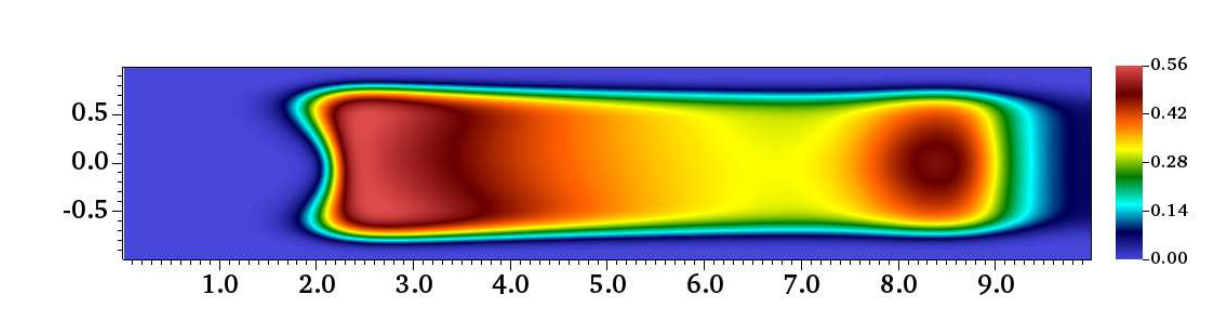}
       \caption{$\rho$ at time $t =16,$ FO-scheme, $\Delta x = \Delta y = 0.0125$} 
     \end{subfigure}\hspace{0.3cm}
           \begin{subfigure}[b]{0.48\textwidth}
         \centering
         \includegraphics[width=\textwidth]{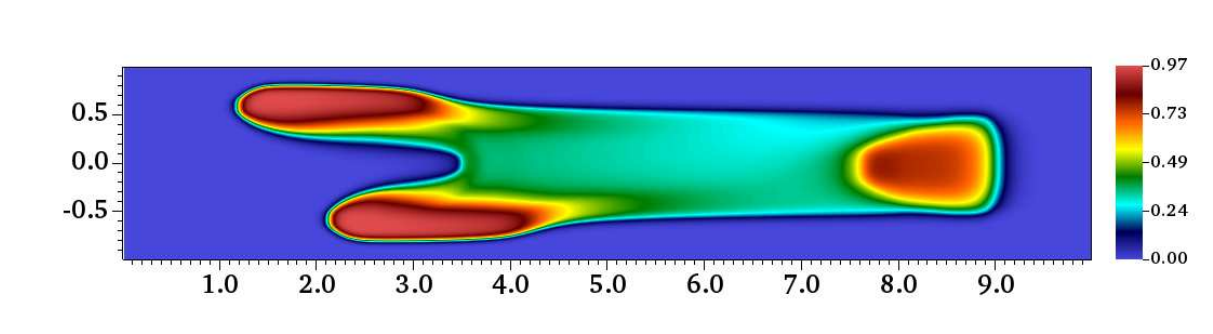}
         \caption{$\rho$ at time $t =16,$ SO-scheme, $\Delta x = \Delta y = 0.0125$}  
     \end{subfigure}

     \begin{subfigure}[b]{0.48\textwidth}
         \centering
         \includegraphics[width=\textwidth]{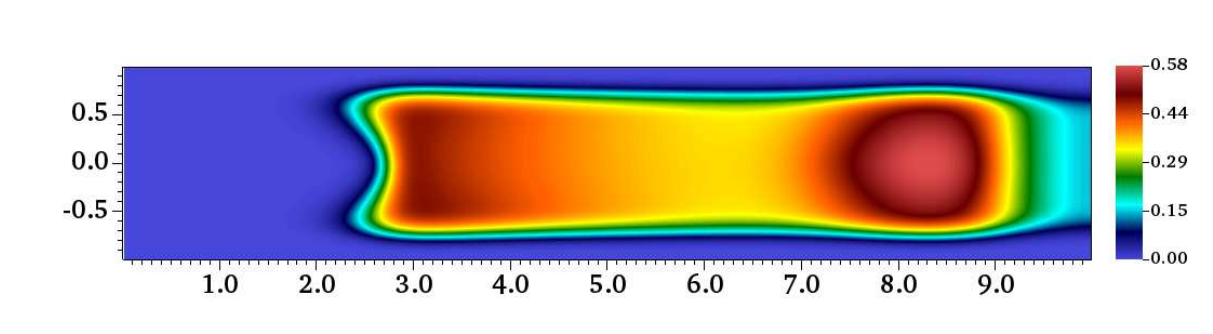}
       \caption{$\rho$ at time $t =20,$ FO-scheme, $\Delta x = \Delta y = 0.0125$} 
     \end{subfigure}\hspace{0.3cm}
           \begin{subfigure}[b]{0.48\textwidth}
         \centering
         \includegraphics[width=\textwidth]{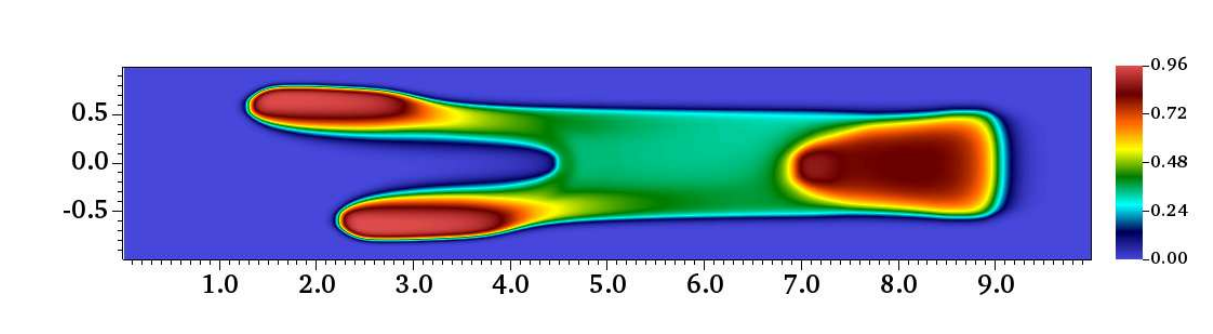}
         \caption{$\rho$ at time $t =20,$ SO-scheme, $\Delta x = \Delta y = 0.0125$}  
     \end{subfigure}
     \caption{Example \ref{ex-cm1}: Profile of approximate solutions $\rho$ of the problem \eqref{crowdproblem} with initial data \eqref{eq:ic} using (a, c, e, g) the  FO scheme and (b, d, f, h) the SO scheme. The time step is taken as \cblue{$\Delta t = 0.026 \Delta x.$}}
     \label{fig:crowd}
\end{figure}
\begin{table}[h!]
\caption{ Example \ref{ex-cm2}. $\mathrm{L}^1$ errors and E.O.C. obtained for the FO and SO schemes to solve the problem \eqref{crowdproblem}  with initial data  \eqref{eq:ic} at time $T=0.2$ with a time step of \cblue{$\Delta t = 0.026 \Delta x.$}}
\centering
\begin{tabular}{|c|l|l|l|l|}
\hline
& \multicolumn{2}{|c|}{FO scheme} & \multicolumn{2}{|c|}{SO scheme} \\ \hline
$h$ & $\|{\rho}_h - {\rho}_{\frac{h}{2}}\|_{{\mathrm{L}}^1}$& $\gamma$  &$\|{\rho}_h - {\rho}_{\frac{h}{2}}\|_{{\mathrm{L}}^1}$ & $\gamma $ \\ \hline
0.05 & 0.63622 &0.3036201& 0.506055 &  0.6217728          \\ \hline               
0.025 & 0.5154761& 0.3999979 & 0.3288709 & 0.7782156  \\ \hline       
0.0125 & 0.3906584 & 0.4629401& 0.1917605 & 0.7862285  \\ \hline     
0.00625 & 0.2834251 & -& 0.1111939 & - \\ \hline 
\end{tabular}
\vspace{0.3cm} 
\label{table:eoa}
\end{table}

\end{example}
\begin{example}\label{ex-kf1} We consider the non-local Keyfitz-Kranzer (KK) system, as introduced in \cite{aggarwal2015}, which extends the classical Keyfitz-Kranzer system from \cite{keyfitz1979} to a non-local framework.
This specific example of the two-dimensional system involves two unknowns, i.e., $N = 2$, with $\displaystyle \boldsymbol{\rho} = (\rho^{1}, \rho^{2}),$ and is given by
\begin{nalign}\label{eq:kksystem}
    \partial_{t}\rho^{1} + \partial_{x}(\rho^{1}\varphi^{1}(\mu * \rho^{1}, \mu*\rho^{2}))+  \partial_{y}(\rho^{1}\varphi^{2}(\mu * \rho^{1}, \mu*\rho^{2})) &= 0,\\
    \partial_{t}\rho^{2} + \partial_{x}(\rho^{2}\varphi^{1}(\mu * \rho^{1}, \mu*\rho^{2}))+  \partial_{y}(\rho^{2}\varphi^{2}(\mu * \rho^{1}, \mu*\rho^{2})) &= 0, 
\end{nalign}where  the functions $\varphi^1$ and $\varphi^2$ are defined as
\begin{equation*}
\begin{aligned}
    \varphi^1(A_1, A_2) &:= \sin(A_1^2 + A_2^2) \quad \mbox{and} \quad 
    \varphi^2(B_1, B_2) :=\cos(B_1^2 + B_2^2).
\end{aligned}
\end{equation*}
\par Here, the kernel function $\mu$ is given by
\begin{equation*}
\mu(x, y) = \frac{\tilde{\mu}(x, y)}{\int \int_{\mathbb{R}^2} \tilde{\mu}},\quad \mbox{ where }
\tilde{\mu} = \left( r^2 - (x^2 + y^2) \right)^3 \chi_{\{(x, y): x^2 + y^2 \leq r^2\}}(x, y)
\end{equation*} and $r$ represents the  radius of the support of $\mu.$
Note that, \eqref{eq:kksystem} fits into the framework of system \eqref{eq:system} with flux functions expressed in the form
\begin{align*}
    f^k(t, x, y,\rho^{k}, \boldsymbol{\eta}*\boldsymbol{\rho}) &:= \rho^{k}\varphi^{1}(\mu * \rho^{1}, \mu*\rho^{2}), \\
    \quad g^k(t, x, y,\rho^{k}, \boldsymbol{\nu}*\boldsymbol{\rho}) &:= \rho^{k}\varphi^{2}(\mu * \rho^{1}, \mu*\rho^{2}), \quad k \in \{1, 2\},
\end{align*}
where the kernel matrices $\boldsymbol{\eta}$ and $\boldsymbol{\nu}$ are given by
\begin{align*}
   \boldsymbol{\eta} = \boldsymbol{\nu} =  \begin{pmatrix}
       \mu & 0\\
       0 & \mu
   \end{pmatrix}
\end{align*}
and 
\begin{align*}
    \boldsymbol{\eta}*\boldsymbol{\rho} &=\boldsymbol{\nu}*\boldsymbol{\rho} = \left(\mu * \rho^{1}, \mu*\rho^{2}\right).
\end{align*}

We conduct numerical simulations   of the problem \eqref{eq:ic_kk} using  the initial condition (see Fig. \ref{fig:ic_kk})
\begin{equation}\label{eq:ic_kk}
\boldsymbol{\rho}_0(x, y)=\left(\rho_{0}^{1}(x,y), \rho_{0}^{2}(x,y)\right) = 
\begin{cases} 
(1, \sqrt{3}), & (x, y) \in (0, 0.4] \times (0, 0.4], \\
(\sqrt{2}, 1), & (x, y) \in [-0.4, 0] \times (0, 0.4], \\
\left( \frac{1}{2}, \frac{1}{3} \right), & (x, y) \in [-0.4, 0] \times [-0.4, 0], \\
(\sqrt{3}, \sqrt{2}), & (x, y) \in (0, 0.4] \times [-0.4, 0], \\
(0, 0), & \text{elsewhere},
\end{cases}
\end{equation}
described in the computational domain $[-1,1]\times [-1,1]$ with {\cblue {out flow boundary conditions  applied along all  boundaries of the domain.}} The radius is set to  $r = 0.0125$ and  the approximate solutions  $(\rho_1, \rho_2)$ are evolved up to time $T=0.1.$ We compare the solutions obtained using the FO and SO schemes for different resolutions, with a common time step of $\Delta t = 0.05\Delta x.$ \cblue{This time-step is determined from \eqref{cfl-2D} using the fact that $\norm{\partial_{\rho}f^{k}},\norm{\partial_{\rho}g^{k}}\leq 1,$ for $k=1,2,$ in this example.} In Fig. \ref{fig:kk}, the FO scheme solutions are computed on a $(1600\times 1600)$ mesh, while  the  SO scheme solutions are computed on a coarser mesh of size $(800\times 800).$ Similarly, in Fig. \ref{fig:kk_b}, we compare the FO scheme solutions on a $(3200\times3200)$ mesh with the SO scheme solutions on a $(1600\times1600)$ mesh. These results show that the SO scheme requires only half of the resolution of the  FO scheme to produce comparable solutions. This highlights the effectiveness of the SO scheme in simulating the given  problem. Furthermore, it is verified that the SO scheme preserves the positivity property and satisfies $\mathrm{L}^\infty$-stability, consistent with the theoretical results.

\begin{figure}
     \centering
          \begin{subfigure}[b]{0.4\textwidth}
         \centering
         \includegraphics[width=\textwidth]{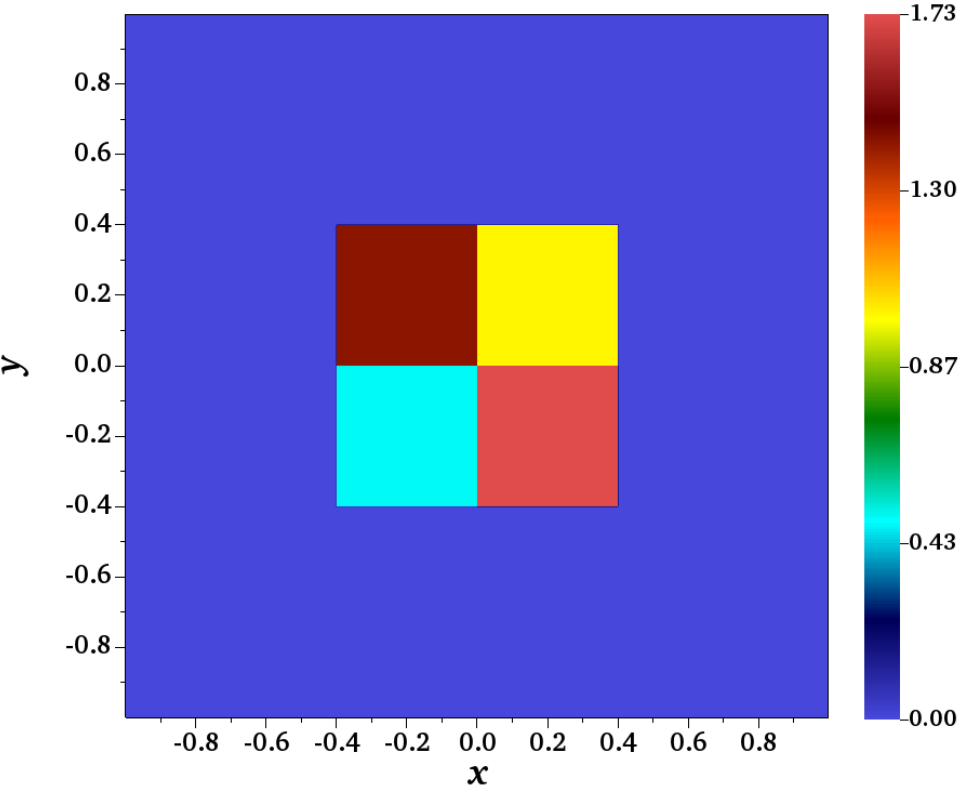}
         \caption{$\rho^1$} 
     \end{subfigure}\hspace{0.5cm}
     \begin{subfigure}[b]{0.4\textwidth}
         \centering
         \includegraphics[width=\textwidth]{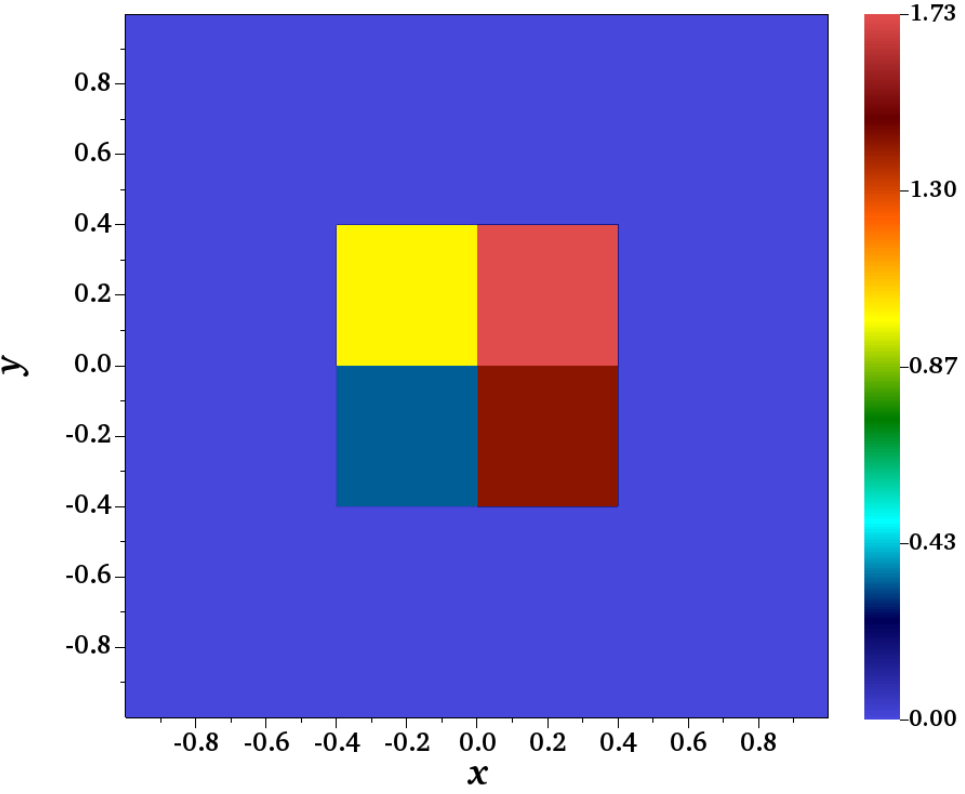}
         \caption{$\rho^2$ } 
     \end{subfigure}
     \caption{Example \ref{ex-kf1}: Initial condition \eqref{eq:ic_kk} for the KK system \eqref{eq:kksystem}.}
     \label{fig:ic_kk}
     \end{figure}

\begin{figure}
     \centering
          \begin{subfigure}[b]{0.4\textwidth}
         \centering
         \includegraphics[width=\textwidth]{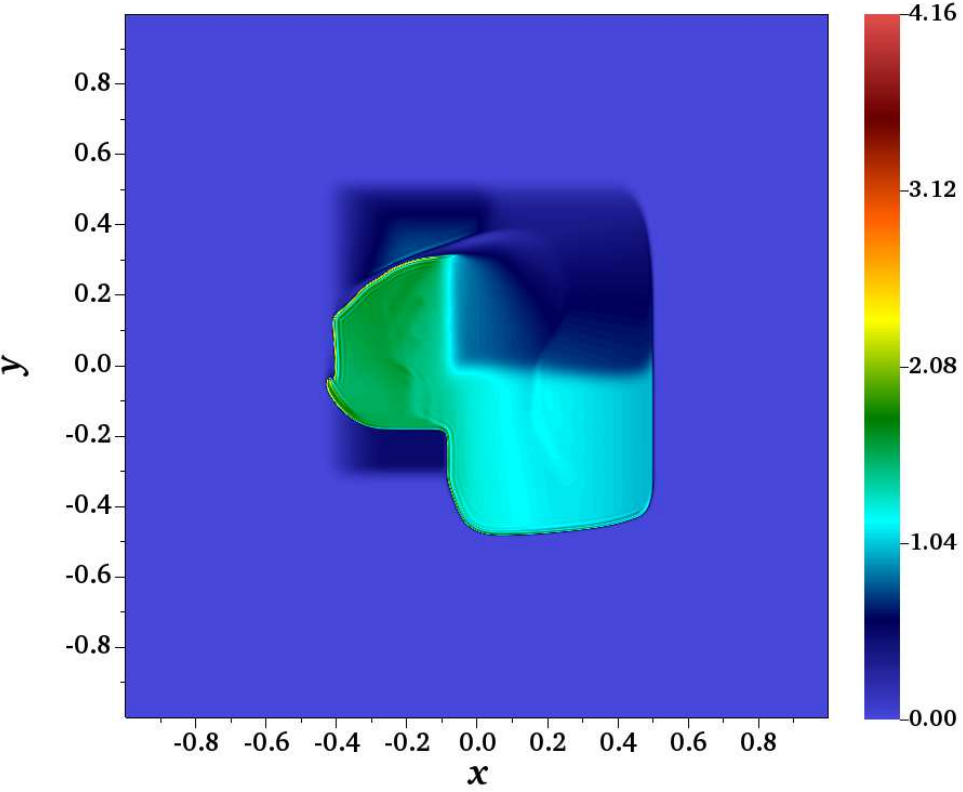}
         \caption{$\rho^1$(FO-scheme, $(1600\times 1600)$ cells)} 
     \end{subfigure}\hspace{0.5cm}
     \begin{subfigure}[b]{0.4\textwidth}
         \centering
\includegraphics[width=\textwidth]{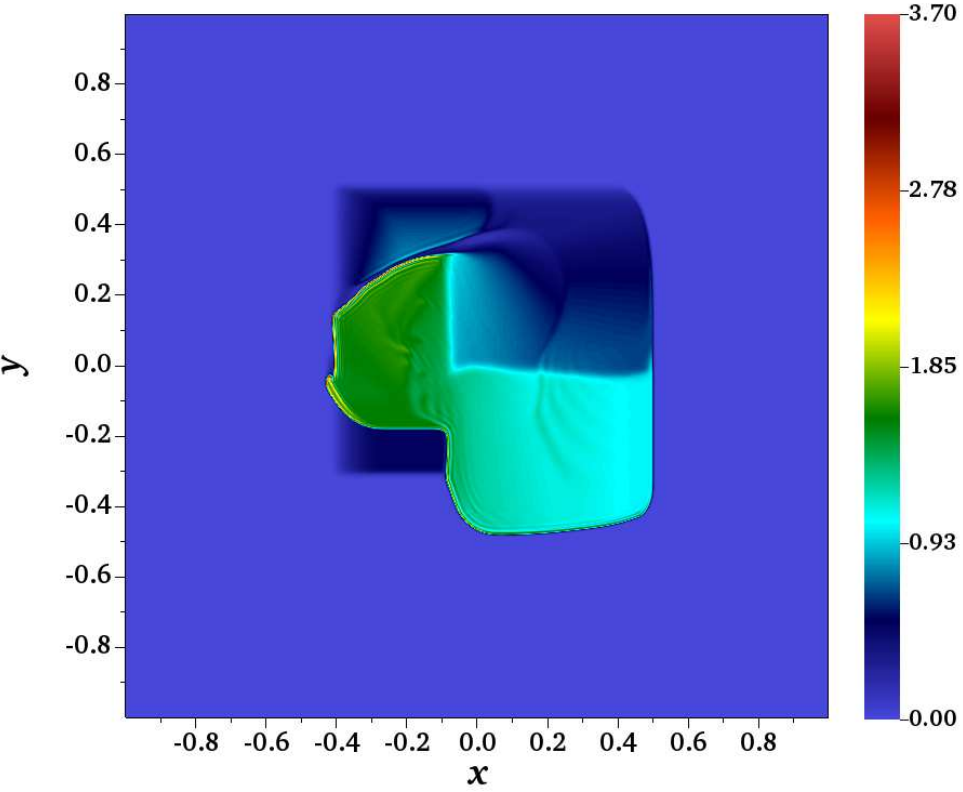}
         \caption{$\rho^1$(SO-scheme, $(800\times 800)$ cells)} 
     \end{subfigure}\\
     
          \begin{subfigure}[b]{0.4\textwidth}
         \centering
         \includegraphics[width=\textwidth]{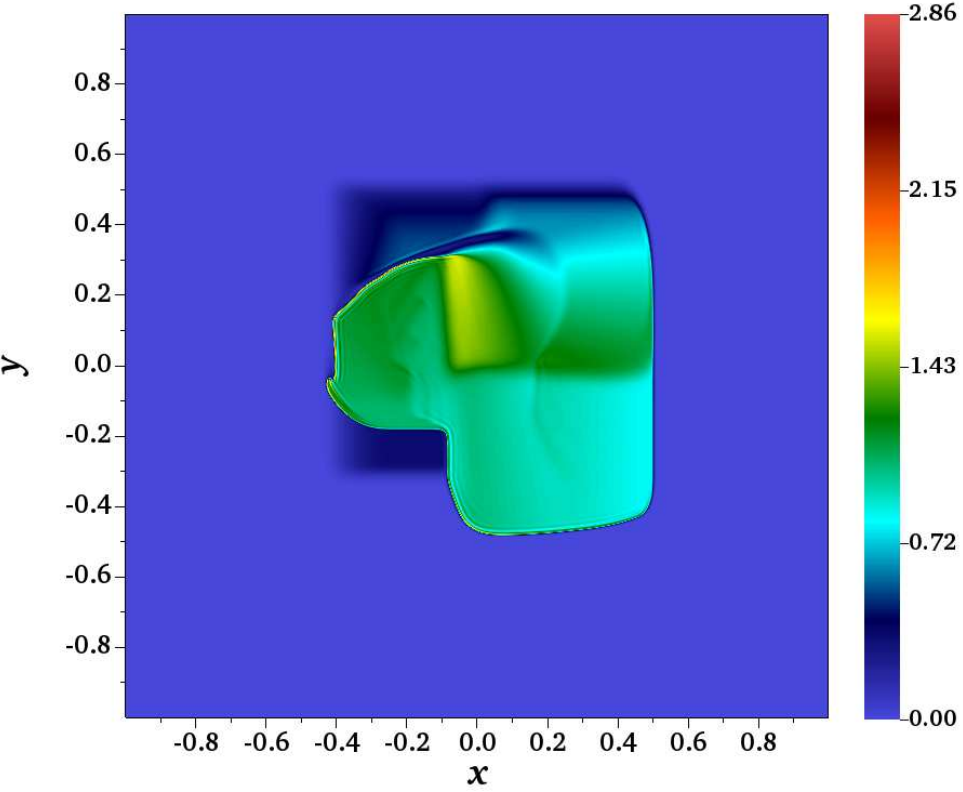}
         \caption{$\rho^2$(FO-scheme, $( 1600\times 1600)$ cells)} 
     \end{subfigure}\hspace{0.5cm}
      \begin{subfigure}[b]{0.4\textwidth}
         \centering
         \includegraphics[width=\textwidth]{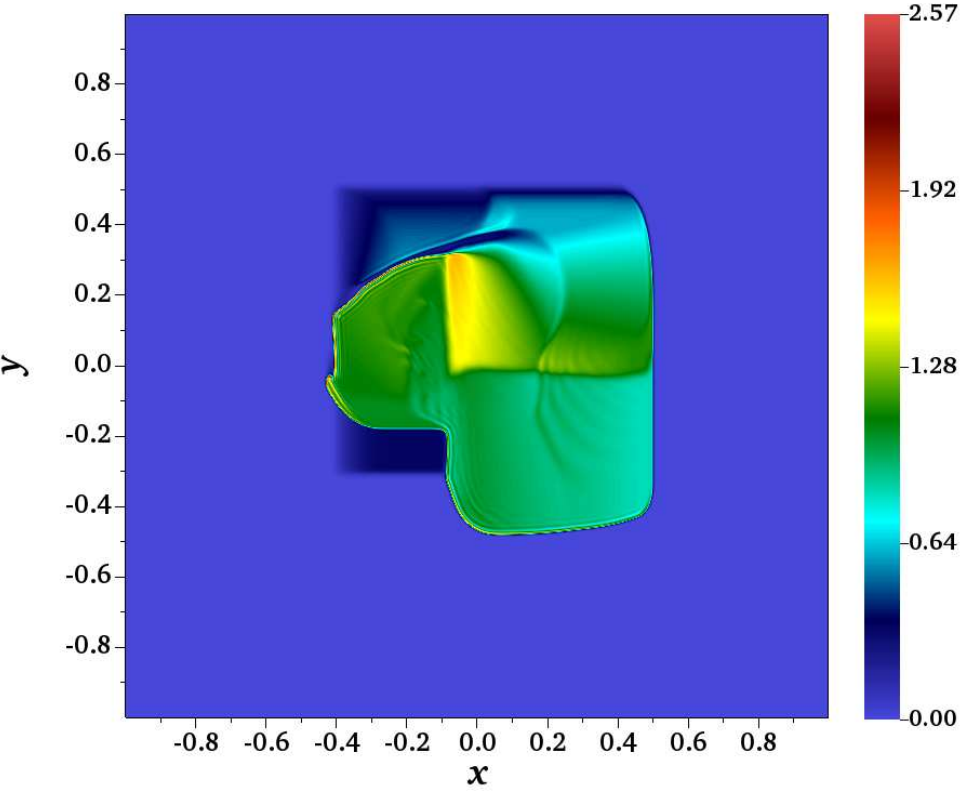}
        \caption{$\rho^2$(SO-scheme, $( 800\times 800)$ cells)} 
     \end{subfigure}
     \caption{Example \ref{ex-kf1}: Numerical solutions $\rho^1$ and $\rho^2$ of the KK system \eqref{eq:kksystem} with the initial condition \eqref{eq:ic_kk}, computed at time $T = 0.1$ using (a, c) the FO scheme and (b, d) the SO scheme. The time step is set as $\Delta t = 0.05 \Delta x,$ and  the parameter of the kernel function is taken as $r = 0.0125.$ FO scheme solutions are computed with a mesh resolution $(1600\times 1600),$ while  SO scheme solutions use a resolution of $(800\times800).$}
     \label{fig:kk}
    \end{figure}

\begin{figure}
     \centering
          \begin{subfigure}[b]{0.4\textwidth}
         \centering
         \includegraphics[width=\textwidth]{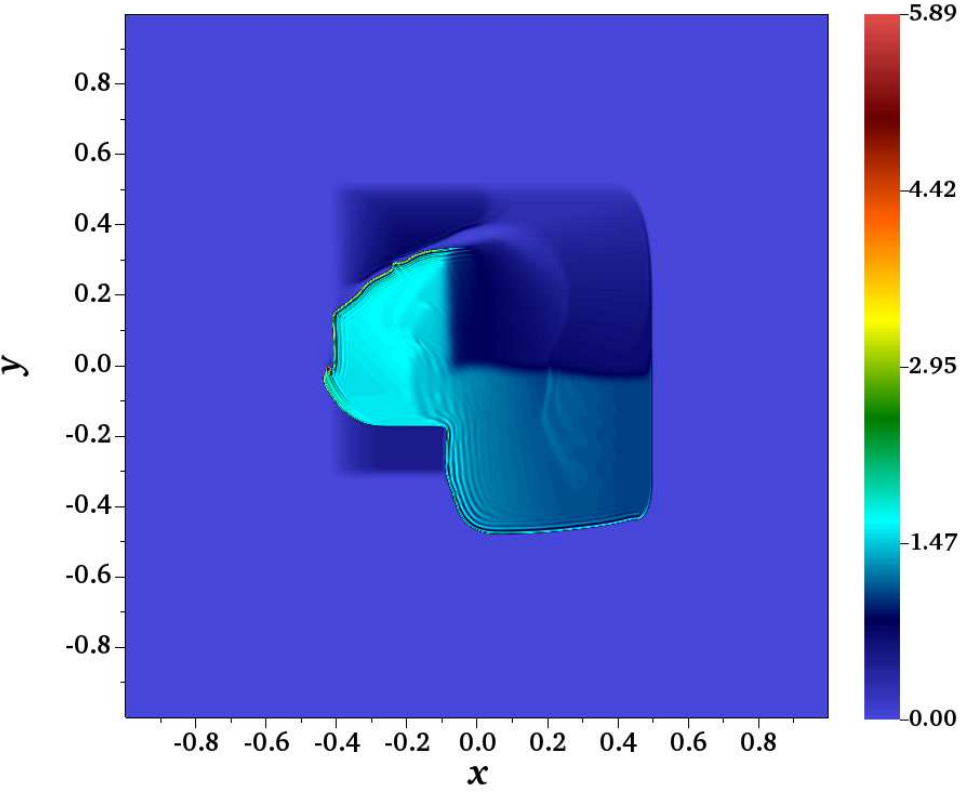}
         \caption{$\rho^1$(FO- scheme,$(3200\times 3200)$ cells) } 
     \end{subfigure} \hspace{0.5cm}
     \begin{subfigure}[b]{0.4\textwidth}
         \centering
         \includegraphics[width=\textwidth]{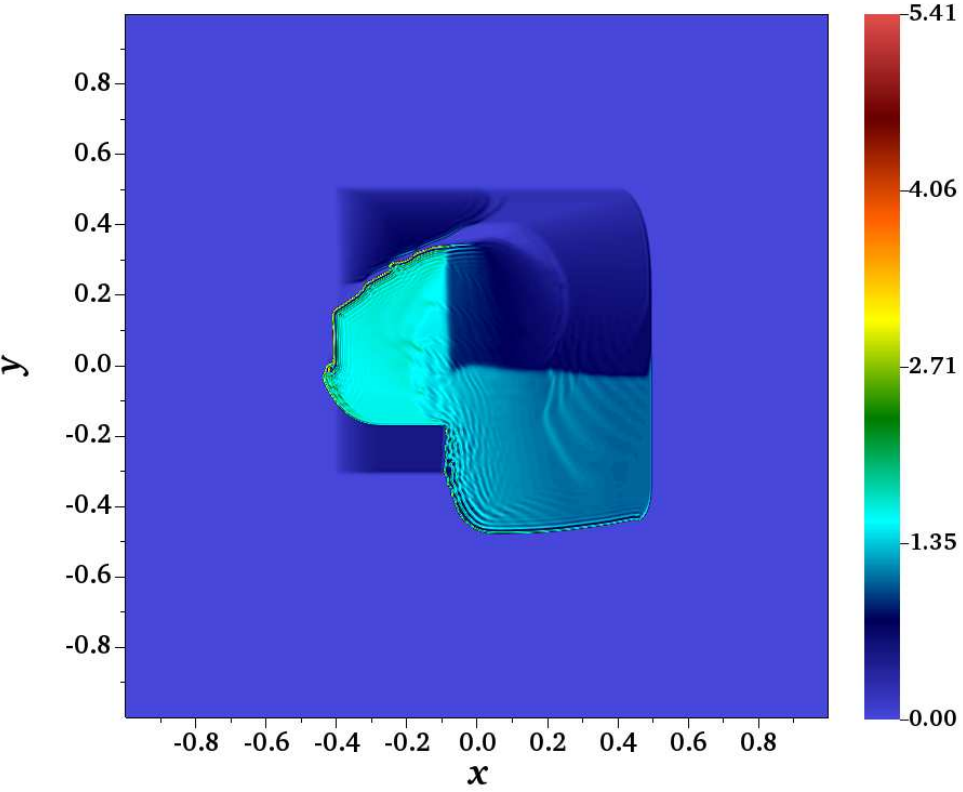}
         \caption{$\rho^1$(SO-scheme, $(1600\times 1600)$ cells)} 
     \end{subfigure}\\
     
          \begin{subfigure}[b]{0.4\textwidth}
         \centering
         \includegraphics[width=\textwidth]{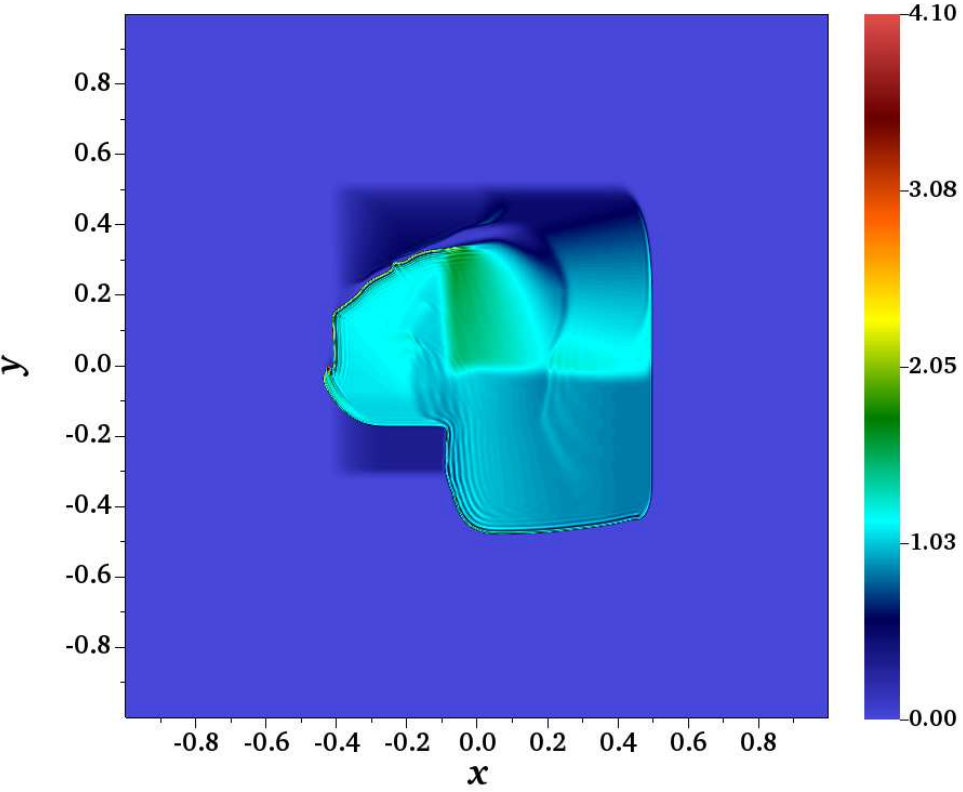}
         \caption{$\rho^2$(FO-scheme,$( 3200\times 3200)$ cells)} 
     \end{subfigure}\hspace{0.5cm}
      \begin{subfigure}[b]{0.4\textwidth}
         \centering
         \includegraphics[width=\textwidth]{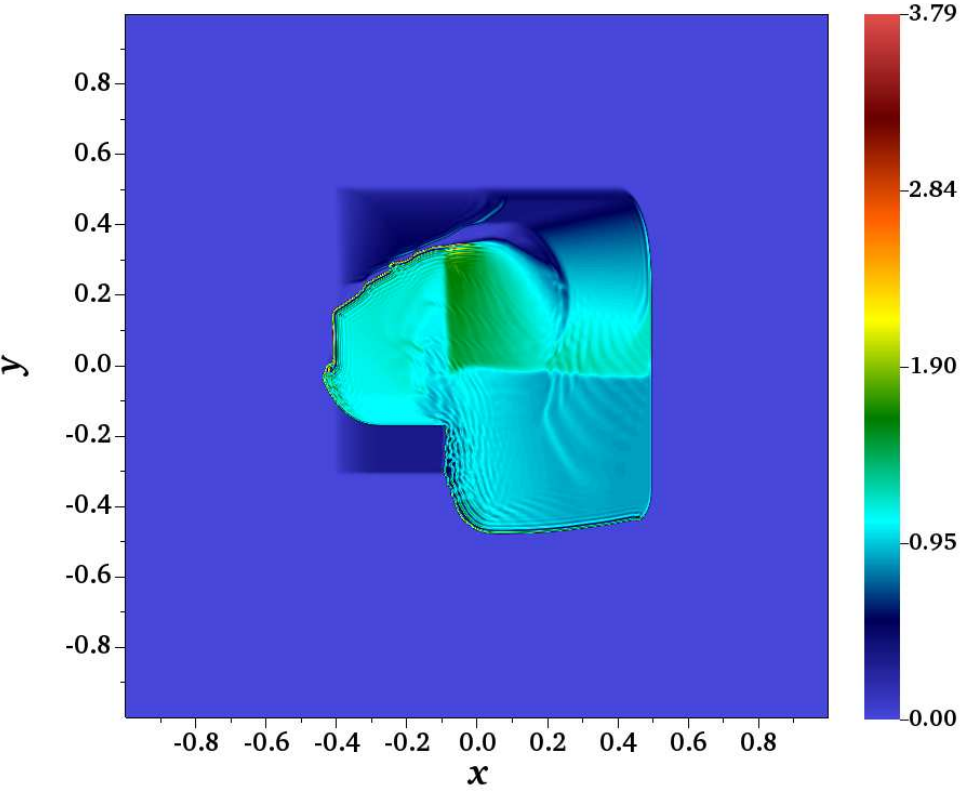}
        \caption{$\rho^2$(SO- scheme,$( 1600\times 1600)$ cells)} 
     \end{subfigure}
     \caption{Example \ref{ex-kf1}: Numerical solutions $\rho^1$ and $\rho^2$ of the KK system \eqref{eq:kksystem} with the initial condition \eqref{eq:ic_kk}, computed at time $T = 0.1$ using (a, c) the FO scheme and (b, d) the SO scheme. The time step is set as \cblue{$\Delta t = 0.05 \Delta x$} and  the parameter of the kernel function is taken as $r = 0.0125.$ FO scheme solutions are computed with a mesh resolution $(3200\times 3200),$ while SO scheme solutions use a resolution of $(1600\times 1600).$ }
         \label{fig:kk_b}
    \end{figure}
     \end{example}
     \begin{example}\label{ex-kf2} In this example, we consider the behavior of  solutions of the non-local Keyfitz-Kranzer model \eqref{eq:kksystem} as the radius of the convolution kernels approaches zero, which is equivalent to  the convolution kernels converging to the Dirac delta distribution. \cblue{This problem, known as the \textquoteleft singular limit problem\textquoteright \, has been investigated numerically in \cite{amorim2015,aggarwal2015}, and theoretical results have been established for specific cases in \cite{coclite2020,colombocrippa2023,colombocrippa2019}. However, analytical convergence results for the general case remain an open problem.}
     It is desirable  that numerical schemes that approximate non-local models retain their robustness under variations in model parameters. A recent study  in this direction is available in \cite{kuang2024}. In view of this, we investigate the behavior of both the FO and SO schemes for the singular limit problem, where the local version of the Keyfitz-Kranzer system is given by:
\begin{nalign}\label{eq:kk_local}
    \partial_{t}\rho^{1} + \partial_{x}(\rho^{1}\varphi^{1}(\rho^{1}, \rho^{2}))+  \partial_{y}(\rho^{1}\varphi^{2}(\rho^{1}, \rho^{2})) &= 0,\\
    \partial_{t}\rho^{2} + \partial_{x}(\rho^{2}\varphi^{1}(\rho^{1}, \rho^{2}))+  \partial_{y}(\rho^{2}\varphi^{2}(\rho^{1}, \rho^{2})) &= 0. 
\end{nalign}
\par 
We perform this analysis using  radii of convolution kernels $r\in \{0.04, 0.02,\\ 0.01, 0.005, 0.0025\}$ across different time levels $t\in\{0.03, 0.07, 0.1\}.$  We compute the $\mathrm{L}^1$ distance between  the solutions corresponding to the non-local \eqref{eq:kksystem} and local \eqref{eq:kk_local} versions of KK system, with the  initial condition specified in \eqref{eq:ic_kk}. All solutions for the non-local problem are computed on a mesh of $(1600\times 1600)$ cells while  the local model \eqref{eq:kk_local} solutions are computed on a finer mesh size of $(3200\times 3200)$ cells with SO scheme. In all the computations, the time step is set as \cblue{$\Delta t = 0.05 \Delta x,$} and the boundary conditions are as in Example \ref{ex-kf1}. The results displayed in Table \ref{table:loc_conv} indicate that  the SO scheme solutions converge to the local version as the parameter $r$ approaches zero. Furthermore,  we observe that the rate of convergence of the SO scheme is higher than that of the FO scheme. 

\begin{table}[h!]
\caption{Example \ref{ex-kf2}: $\mathrm{L}^1$ distance between the solutions corresponding to the non-local \eqref{eq:kksystem} and local \eqref{eq:kk_local} versions of KK system with initial condition \eqref{eq:ic_kk} for the  FO and SO schemes, computed on a mesh of resolution $(1600\times 1600)$. The solutions of the local problem  are computed  with a mesh of $(3200 \times 3200)$ cells using the SO scheme. The  kernel radii are chosen as $r\in \{0.04, 0.02, 0.01, 0.005, 0.0025\},$ solutions are computed at times $t\in\{0.03, 0.07, 0.1\}$ and the time step is set as $\Delta t = 0.05 \Delta x.$}
\centering
\begin{tabular}{|c|c|l|l|l|l|l|l|l|l|}
\cline{1-8}
\multirow{2}{*}{Scheme}& & \multicolumn{3}{|c|}{$\rho^1$} & \multicolumn{3}{|c|}{$\rho^2$}\\ \cline{2-8}
&\diagbox[]{$r$}{$t$} & 0.03 & 0.07 & 0.1  &0.03 & 0.07 & 0.1 \\ \cline{1-8}
\multirow{5}{*}{FO}&  0.04 &  0.0937 & 0.1446 & 0.1575 & 0.0837 & 0.1344  & 0.1376\\ \cline{2-8}
& 0.02 & 0.0576 &0.0836 & 0.0882 & 0.0519 & 0.0790 & 0.0808\\ \cline{2-8}
& 0.01 & 0.0384 & 0.0519  & 0.0531 & 0.0344 & 0.0493 & 0.0496 \\  \cline{2-8}
&0.005  & 0.0250 & 0.0317 & 0.0317 & 0.0225  & 0.0297  & 0.0293 \\ \cline{2-8}
&0.0025  & 0.0179  & 0.0226  & 0.0239  & 0.0169  & 0.0208 & 0.0216 \\ \cline{1-8}
\multirow{5}{*}{SO}&  0.04 & 0.1323  & 0.2379  &  0.2839  &0.1223  & 0.2262 & 0.2586 \\ \cline{2-8}
&0.02 & 0.0843 & 0.1373 & 0.1511 & 0.0789 & 0.1327 & 0.1392 \\ \cline{2-8}
& 0.01 & 0.0492  & 0.0749  & 0.0807 & 0.0462  & 0.0750  & 0.0787  \\ \cline{2-8}
& 0.005 & 0.0288 & 0.0390 & 0.0410  & 0.0263  & 0.0389  & 0.0395  \\ \cline{2-8}
&0.0025 & 0.0115 & 0.0138 & 0.0137 & 0.0100 & 0.0133 &   0.0129 \\ \cline{1-8}
\end{tabular}
\vspace{0.5cm}
\label{table:loc_conv}
\end{table}

\end{example}

\section{Conclusion}\label{section:conclusion}
In this work, we propose a fully discrete second-order scheme for a general system of non-local conservation laws in multiple dimensions. The resulting scheme is theoretically shown to satisfy the positivity-preserving property and proven to be  $\mathrm{L}^\infty$-stable. Numerical experiments clearly indicate the superiority of the SO scheme over its first-order counterpart, as illustrated in Figs. \ref{fig:crowd} and \ref{fig:kk} for both scalar and system cases. We have also shown the numerical convergence of the SO scheme in the scalar case and compared \cblue{it} to \cblue{that of} the FO scheme, see Table \ref{table:eoa}.  The robustness of the SO scheme is further evaluated  in the context of the  \textquoteleft singular limit problem\textquoteright  \, and the results show that the SO scheme solutions approach the local problem as the parameter $r$ tends to zero, with a higher convergence rate compared to  that of FO scheme, as is evident from Table \ref{table:loc_conv}. 
Additionally, we wish to note that a key challenge in analyzing the theoretical convergence of the second-order scheme lies in deriving a bounded variation estimate. To the best of our knowledge, such estimates are unavailable in the literature for the case of local conservation laws as well, except \cblue{for the weak-BV estimates} in \cite{Coquel1991}. We plan to address these theoretical aspects in a future work.

%
\newpage

\bibliographystyle{spmpsci}      
\bibliography{ref}
%
%

\end{document}